\documentclass[leqno,12pt]{article}

\def\today{${\scriptscriptstyle\number\day-\number\month-\number\year}$}
\usepackage{graphicx}
\usepackage{fancyhdr}
\usepackage{amsmath}
\usepackage{amsthm}
\usepackage{amsfonts}
\input amssym.def
\input amssym.tex
\newtheorem{theorem}{Theorem}[section]
\newtheorem*{thm}{Main Theorem}

\newtheorem{lemma}[theorem]{Lemma}

\newtheorem{definition}[theorem]{Definition}

\newtheorem{remark}[theorem]{Remark}

\def\address#1{{\center{#1}}}
\date{}
\def\m@th{\mathsurround=0pt}
\def\eqal#1{\null\,\vcenter{\openup\jot\m@th
 \ialign{\strut\hfil$\displaystyle{##}$&&$\displaystyle{{}##}$\hfil
 \crcr#1\crcr}}\,}
\def\matrix#1{\null\,\vcenter{\normalbaselines\m@th
 \ialign{\hfil$##$\hfil&&\quad\hfil$##$\hfil\crcr
 \mathstrut\crcr\noalign{\kern-\baselineskip}
 #1\crcr\mathstrut\crcr\noalign{\kern-\baselineskip}}}\,}
\def\lc#1{\hbox to .89\hsize{\qquad$\displaystyle{#1}\hfill$}}

\newcommand{\benn}{\begin{eqnarray*}}
\newcommand{\eenn}{\end{eqnarray*}}
\newcommand{\ben}{\begin{eqnarray}}
\newcommand{\een}{\end{eqnarray}}
\newcommand{\bal}{\begin{aligned}}
\newcommand{\eal}{\end{aligned}}
\def\D{{\mathbb D}}
\def\I{{\mathbb I}}
\def\N{{\mathbb N}}
\def\R{{\mathbb R}}
\def\T{{\mathbb T}}

\def\de{\delta}
\def\Om{\Omega}
\def\om{\omega}
\def\te{\theta}

\def\te{\theta}

\def\div{{\rm div}\,}
\def\divv{{\rm div}\,}

\def\nb{\nabla}

\def\dist{{\rm dist\,}}
\def\supp{{\rm supp\,}}
\def\const{{\rm const}}

\numberwithin{equation}{section}

\title{Weak solutions to incompressible heat-conducting motions with large flux}
\author{Joanna Renc\l awowicz$^{1)}$\quad Wojciech M. Zaj\c{a}czkowski$^{1), 2)}$}

\begin{document}

\maketitle

\address{$^{1)}$\ Institute of Mathematics, Polish Academy of Sciences,\\
\'Sniadeckich 8, 00-656 Warsaw, Poland\\ e-mail: jr@impan.pl, ORCID: 0000-0001-8597-9366, \\ {\it corresponding author,} \\
$^{2)}$\ Institute of Mathematics and Cryptology, Cybernetics Faculty, \\
Military University of Technology,\\
S. Kaliskiego 2, 00-908 Warsaw, Poland\\
e-mail: wz@impan.pl, ORCID: 0000-0003-1229-2162, \\ {\it emeritus professor at IMPAN} \\}

\begin{abstract}
We analyze the problem for viscous incompressible heat-conduc\-ting fluid in a finite cylinder with large inflow and outflow, modelled with Navier-Stokes equations coupled with the heat equation. We prove energy estimate without restrictions on the magnitudes of the external force, initial data, inflow and outflow. In order to estimate nonlinear terms we use weighted spaces. Next, using Galerkin approximation, we show existence of weak solutions.
\end{abstract}
\let\thefootnote\relax\footnote{
{\it Mathematics Subject Classification (MSC2020)}: Primary 35Q30; Secondary 76D03, 76D05  }

{\bf Abbreviated title:} Weak solutions for heat-conducting flow.

{\bf Statement relating to the ethics and integrity policies:}

\noindent On behalf of all authors, the corresponding author states that the manuscript is not published elsewhere and there is no conflict of interest.

\noindent The datasets generated during and/or analysed during the current study are available from the corresponding author on reasonable request.

{\bf Funding statement:} There are no funders to report for this submission.

\section{Introduction}\label{s1}

We consider the following motions of viscous incompressible heat-condu\-cting fluid in a finite cylinder $\Omega\subset\R^3$ with boundary $S=S_1\cup S_2$ with arbitrary large inflow and outflow
\begin{equation} \eqal{
& v_t+v\cdot\nabla v-\divv\T(v,p)=\omega(\theta)f \ \  {\rm in}\ \ \Omega\times\R_+, \cr
& \divv v=0\qquad {\rm in}\ \ \Omega\times\R_+, \cr
& \theta_t+v\cdot\nabla\theta-\varkappa\Delta\theta=0\qquad {\rm in}\ \ \Omega\times\R_+, \cr
& \nu\bar n\cdot\D(v)\cdot\bar\tau_\alpha+\gamma v\cdot\bar\tau_\alpha=0,\ \ \alpha=1,2,\qquad {\rm on}\ \ S_1\times\R_+, \cr
& v\cdot\bar n=0\qquad {\rm on}\ \ S_1\times\R_+, \cr
& \bar n\cdot\D(v)\cdot\bar\tau_\alpha=0,\ \ \alpha=1,2,\qquad {\rm on}\ \ S_2\times\R_+, \cr
& v\cdot\bar n=d\qquad {\rm on}\ \ S_2\times\R_+, \cr
& \bar n\cdot\nabla\theta=0\qquad {\rm on}\ \ S_1\times\R_+, \cr
& \bar n\cdot\nabla\theta=0\qquad {\rm on}\ \ S_2\times\R_+, \cr
& v|_{t=0}=v_0,\ \ \theta|_{t=0}=\theta_0\qquad {\rm in}\ \ \Omega, \cr}
\label{1.1}
\end{equation}
where $\Omega$ is parallel to the $x_3$-axis of the Cartesian coordinates\break $x=(x_1,x_2,x_3)$. Moreover, $S_1$ is parallel to the $x_3$-axis and $S_2$ is perpendicular to it, $v=(v_1(x,t),v_2(x,t),v_3(x,t))\in\R^3$ is the velocity of the fluid, $p=p(x,t)\in\R$ is the pressure, $\theta=\theta(x,t)\in\R_+$ is the temperature of the fluid, $f=(f_1(x,t),f_2(x,t),f_3(x,t))\in\R^3$ is the external force field, $\omega\in\R^1$ is the function of temperature, $\bar n$ is the unit outward vector normal to $S$ and $\bar\tau_\alpha$, $\alpha=1,2$, are tangent vectors to $S$. By $\T(v,p)=\nu\D(v)-p\I$ we denote the stress tensor, where $\nu>0$ is the constant viscosity coefficient, $\I$ is the unit matrix, and $\D(v)=\{v_{i,x_j}+v_{j,x_i}\}_{i,j=1,2,3}$ is the dilatation tensor. Additionally, $\varkappa>0$ is the constant heat conductivity coefficient, $\gamma>0$ is the slip coefficient.

Let $a$, $c_0$ be positive numbers and let $\phi_0(x_1,x_2)=c_0$ be a sufficiently smooth closed curve (not a circle) in the plane $x_3=\const\in[-a,a]$. Then
$$\eqal{
&S_1=\{x\in\R^3\colon\phi_0(x_1,x_2)=c_0,-a<x_3<a\},\cr
&S_2=S_2(-a)\cup S_2(a),\cr
&S_2(-a)=\{x\in\R^3\colon\phi_0(x_1,x_2)<c_0,x_3=-a\},\cr
&S_2(a)=\{x\in\R^3\colon\phi_0(x_1,x_2)<c_0,x_3=a\}.\cr}
$$
To describe inflow and outflow we denote $d=(d_1, d_2),$
$$
d|_{S_2(a_1)}=-d_1,\quad d|_{S_2(a_2)}=d_2,
$$
where $d_i\ge 0$, $i=1,2$, $a_1=-a$, $a_2=a$.

The continuity equation $(\ref{1.1})_2$ implies the compatibility condition
\begin{equation}
\intop_{S_2(-a)}d_1dS_2=\intop_{S_2(a)}d_2dS_2.
\label{1.11}
\end{equation}

In this paper we show the energy estimate to problem (\ref{1.1}) without restrictions on the
magnitudes of the external force $f,$ initial data $v(0),\theta(0),$ inflow
$d_1$ and outflow $d_2.$ We assume that the flux does not have to vanish as $t \rightarrow
\infty.$  The result requires applying the Hopf function (see \cite{L1}) and weighted estimates
proved in \cite{RZ2,RZ3,RZ4} and recalled in Lemma~\ref{l2.2}. Namely, we need to estimate nonlinear term caused by $v\cdot \nb v$ multiplied by test function in this way to get some smallness and to absorb this small term with norms on the l.h.s. Since the norm in $L_p$ is not small, we use weighted norm, see Remark~\ref{r3.3} for details.

Now, we formulate the main result, established in Theorems~\ref{l3.1} and \ref{l5.8}.
\begin{thm}\label{t1} (Weak solutions and energy estimate)

\noindent Let $(v,p,\theta)$ be a solution to problem (\ref{1.1}). Assume that \\
$d_i\in L_\infty(0,T;W_3^1(S_2(a_i)))\cap L_2(0,T;H^{1/2}(S_2(a_i))), i=1,2,$ $\theta(0)\in L_2(\Omega),$ $\omega f\in L_{6/5}(\Om^T).$ 
Then there exist weak solution $(w, \theta)$, where $ w$ is defined by (\ref{4.9}), $\vartheta$ is defined by (\ref{2-6})
and 
\ben \bal
& \|v\|_{V(\Omega^T)}^2 + \|\theta\|_{V(\Omega^T)}^2 \le  A_1^2+\|w(0)\|_{L_2(\Omega)}^2 
+ \|\vartheta(0)\|_{L_2(\Omega^t)}^2 \\ & +c\sup_{t'\in(0,T)}\|\tilde d(t')\|_{L_2(\Omega)}^2 +c\intop_{0}^T\|\tilde d(t')\|_{W^1_2(\Omega)}^2dt'\\ &  +c\sup_{t'\in(0,T]}\|\tilde d(t')\|_{W_{3,\infty}^1(\Omega)}^4\exp\bigg( \frac{1}{c}\sup_{t'\in(0,T)}\|\tilde d(t')\|_{W_{3,\infty}^1(\Omega)}\bigg)\cdot\\
&\quad\cdot\intop_{0}^T\sum_{i=1}^2\sup_{x_3}\|\tilde d_i(t')\|_{L_2(S_2(a_i))}^2dt'+c\|\om(\te)f\|_{L_{6/5}(\Omega)}^2,
\label{1-3} \eal
\een
 where $\tilde d_i$ is an extension of $d_i$ on $\Omega$, $\tilde{d}_i|_{S_2(a_i)} = d_i,$ $i=1,2,$  
 $d$ denotes $(d_1,d_2),$
\benn ||v||_{V(\Om^T)} \equiv {\rm ess}
\sup_{t\in(0,T)}||v||_{L_2(\Om)} + \left(\int_0^T
||\nb v||_{L_2(\Om)}^2 dt \right)^{1/2} \eenn 
and 
\benn
A_1^2=\int_{0}^{T}(\|\omega(\theta)f(t)\|_{L_{6/5}(\Omega)}^2+ \sum_{i=1}^2\|d_i(t)\|_{W^1_3(S_2(a_i))}^2+\sum_{i=1}^2\|d_{i,t}(t)\|_{W^1_{6/5}(S_2(a_i))}^2)dt \cdot\\
\quad\, \,\cdot\phi(\sum_{i=1}^2\sup_t\|d_i(t)\|_{W^1_3(S_2(a_i))})<\infty,
\eenn
with an increasing positive function $\phi.$ 
Moreover, $\theta$ is bounded from above and below with positive $\theta^*$ and $\theta_*,$ respectively, $\theta_* < \theta^*$
\ben \bal
&\|\theta\|_{V(\Omega^T)}^2 \le c\exp(\|d_1\|_{L_6(0,T;L_3(S_2(a_1)))}^6)\|\theta(0)\|_{L_2(\Omega)}^2
\\ & \quad
 + \|d_1\|^2_{L_1(S_2^T(a_1))}{\theta^*}^2+ \|\theta(0)\|_{L_2(\Om)}^2.
\label{1-4}
\eal \een
\end{thm}
We show the existence of weak solutions to problem (\ref{1.1}) through the Galerkin method, similarly as in \cite{ZZ}. On the other hand, comparing results with \cite{ZZ}, we can relax conditions on function $\omega(\theta)$ since we prove additionally boundedness for temperature $\theta$ from below an above, see Lemma~\ref{l2.4}.

\section{Notation and auxiliary results}\label{s2}

First we introduce a simplified notation of spaces used in this paper.

\begin{definition}[see \cite{Ad}]\label{d2.1}
For Lebesque and Sobolev spaces we set the following notation
$$\eqal{
&\|u\|_{L_p(Q)}=|u|_{p,Q},\quad \|u\|_{L_p(Q^t)}=|u|_{p,Q^t},\cr
&\|u\|_{L_q(0,t;L_p(Q))}=|u|_{p,q,Q^t},\cr}
$$
where $Q$ is equal either $\Omega$, or $S=\partial\Omega$, or $\R^3$, $Q^t=Q\times(0,t)$, $p,q\in[1,\infty]$.

Let $H^s(Q)=W_2^s(Q)$, $s\in\N_0=\N\cup\{0\}$. Then we denote
$$\eqal{
&\|u\|_{H^s(Q)}=\|u\|_{s,Q},\quad &\|u\|_{W_p^s(Q)}=\|u\|_{s,p,Q},\cr
&\|u\|_{L_q(0,t;W_p^s(Q))}=\|u\|_{s,p,q,Q^t},\quad &\|u\|_{L_q(0,t;W_q^s(Q))}=\|u\|_{s,q,Q^t}.\cr}
$$
\end{definition}

Introduce weighted spaces $L_{p,\mu}(\Omega), L^k_{p,\mu}(\Omega)$, $p\in[1,\infty]$, $\mu\in(0,1)$, with the norms
$$
\|u\|_{L_{p,\mu}(\Omega)}=\bigg(\intop_\Omega|u|^p\eta^{p\mu}dx\bigg)^{1/p},
$$
\benn
\|u\|_{L_{p,\mu}^k(\Omega)}=\bigg(\sum_{|\alpha|=k}\int_\Omega |D_x^\alpha u|^p\eta^{p\mu}dx\bigg)^{1/p}<\infty,
\eenn
where $\eta(x_3)=\min_{i=1,2}\dist(x_3,S_2(a_i)).$ 

\noindent We define also
\benn V(\Om^t) & \equiv & V_2^0(\Om^T) = \{u: ||u||_{V^0_2(\Om^T)} = {\rm ess}
\sup_{t\in(0,T)}||u||_{L_2(\Om)} \\ & & + \left(\int_0^T
||\nb u||_{L_2(\Om)}^2 dt \right)^{1/2} <\infty \}, \\
V_2^1(\Om^T) & \equiv & \{u: ||u||_{V^1_2(\Om^T)} = {\rm ess}
\sup_{t\in(0,T)}||u||_{H^1(\Om)} \\ & & + \left(\int_0^T ||\nb
u||_{H^1(\Om)}^2 dt \right)^{1/2} <\infty \}. \eenn 

Consider the Neumann problem (see (\ref{4.7}))
\begin{equation}\eqal{
&\Delta\varphi=-\divv b\quad &{\rm in}\ \ \Omega,\cr
&\bar n\cdot\nabla\varphi=0\quad &{\rm on}\ \ S,\cr
&\intop_\Omega\varphi dx=0,\cr
& \int_{\Omega} \divv b = 0. \cr}
\label{2.1}
\end{equation}

\begin{lemma}[see \cite{RZ3}, \cite{RZ4}]\label{l2.2}
Let $\divv b\in L_{p,\mu}(\Omega)$, $p\ge 2$, $\mu\in(0,1)$. Then the weak solution to (\ref{2.1}) satisfies the estimate
\begin{equation}
\|\nabla^2\varphi\|_{L_{p,\mu}(\Omega)}\le c\|\divv b\|_{L_{p,\mu}(\Omega)}.
\label{2.2}
\end{equation}
\end{lemma}

In what follows we use the anisotropic Sobolev spaces $W_{p_1,p_2}^1(\Omega)$ and $L_{p_1,p_2}(\Omega)$, where $p_1,p_2\in[1,\infty]$ with the norms
$$
\|u\|_{W_{p_1,p_2}^k(\Omega)}=\bigg[\bigg(\intop_{S_2}\sum_{|\alpha|\le k}|D_x^\alpha u|^{p_1}dS_2\bigg)^{p_2/p_1}dx_3\bigg]^{1/p_2},
$$
where $k\in\N\cup\{0\}$ and $W_{p_1,p_2}^0(\Omega)=L_{p_1,p_2}(\Omega)$.

Consider the problem
\begin{equation}\eqal{
&\theta_t+v\cdot\nabla\theta-\varkappa\Delta\theta=0,\cr
&\divv v=0,\ \ v\cdot\bar n|_{S_1}=0,\ \ v\cdot\bar n|_{S_2}=d,\ \ \bar n\cdot\nabla\theta|_S=0,\cr
&\theta|_{t=0}=\theta_0.\cr}
\label{2.3}
\end{equation}

\begin{lemma}\label{l2.3}
Let $\inf_{x\in\Omega}\theta_0(x)=\theta_*>0$ and $d_1 \in L_2(0,t; L_{\infty}(S_2(-a)).$\\
Then solutions to (\ref{2.3}) satisfy
\begin{equation}
\theta\ge\theta_*.
\label{2.4}
\end{equation}
\end{lemma}

\begin{proof}
Multiply $(\ref{2.3})_1$ by $(\theta-\theta_*)_-=\min\{\theta-\theta_*,0\}$ and integrate over $\Omega$. Then we obtain
\begin{equation}\eqal{
&\frac{1}{2}\frac{d}{dt}\intop_\Omega(\theta-\theta_*)_-^2dx+\varkappa |\nabla(\theta-\theta_*)_-|^2dx\cr
&=-\intop_\Omega v\cdot\nabla(\theta-\theta_*)_-(\theta-\theta_*)_-dx\cr}
\label{2.5}
\end{equation}
The r.h.s. of (\ref{2.5}) equals
$$\eqal{
&-\frac{1}{2}\intop_\Omega v\cdot\nabla(\theta-\theta_*)_-^2dx=-\frac{1}{2}\intop_{S_2}v\cdot\bar n(\theta-\theta_*)_-^2dS_2\cr
&=\frac{1}{2}\intop_{S_2(-a)}d_1(\theta-\theta_*)_-^2dS_2-\frac{1}{2}\intop_{S_2(a)}d_2(\theta-\theta_*)_-^2dS_2\cr}
$$
Hence, (\ref{2.5}) takes the form
\begin{equation}\eqal{
&\frac{1}{2}\frac{d}{dt}|(\theta-\theta_*)_-|_{2,\Omega}^2+\varkappa|\nabla(\theta-\theta_*)_-|_{2,\Omega}^2\cr
&\le\frac{1}{2}\intop_{S_2(-a)}d_1(\theta-\theta_*)_-^2dS_2\le\frac{1}{2}|d_1|_{\infty,S_2(-a)}\intop_{S_2(-a)}(\theta-\theta_*)_-^2dS_2\cr
&\le\frac{1}{2}|d_1|_{\infty,S_2(-a)}(\varepsilon|\nabla(\theta-\theta_*)_-|_{2,\Omega}^2+ \frac{c}{\varepsilon}|(\theta-\theta_*)_-|_{2,\Omega}^2),\cr}
\label{2.6}
\end{equation}
where we used some interpolation. For $\varepsilon$ sufficiently small we obtain
\begin{equation}\eqal{
&\frac{d}{dt}|(\theta-\theta_*)_-|_{2,\Omega}^2+\varkappa|\nabla(\theta-\theta_*)_-|_{2,\Omega}^2\cr
&\le c|d_1|_{\infty,S_2(-a)}^2|(\theta-\theta_*)_-|_{2,\Omega}^2.\cr}
\label{2.7}
\end{equation}
Dropping the second term on the l.h.s. of (\ref{2.7}) and integrating with respect to time yield
\begin{equation}\eqal{
&\intop_\Omega(\theta(t)-\theta_*)_-^2dx\le \exp\bigg(c\intop_0^t|d_1(t')|_{\infty,S_2(-a)}^2dt'\bigg)\cdot\cr
&\quad\cdot\intop_\Omega(\theta(0)-\theta_*)_-^2dx.\cr}
\label{2.8}
\end{equation}
Since $\theta(0)>\theta_*$ we have that $(\theta(0)-\theta_*)_-=0$ so also $(\theta(t)-\theta_*)_-=0$. This proves (\ref{2.4}) and concludes the proof.
\end{proof}

\begin{lemma}\label{l2.4}
Let $\theta_*$, $\theta^*$ be positive constants such that $\theta_*<\theta^*$ and $\theta_*\le\theta_0\le\theta^*$. Let $d_1 \in L_2(\R_+, L_{\infty}(S_2(-a))) $ Then solutions to (\ref{2.3}) satisfy
\begin{equation}
\theta_*\le\theta(t)\le\theta^*.
\label{2.9}
\end{equation}
\end{lemma}

\begin{proof}
Multiply $(\ref{2.3})_1$ by $\theta^{s-1}, s>1$ and integrate over $\Omega$. Then we obtain
$$\eqal{
&\frac{1}{s}\frac{d}{dt}|\theta|_{s,\Omega}^s+\frac{4\varkappa(s-1)}{s^2}\intop_\Omega|\nabla\theta^{s/2}|^2dx\cr
&=-\frac{1}{s}\intop_\Omega v\cdot\nabla\theta^sdx=-\frac{1}{s}\intop_{S_2}d\theta^sdS_2\cr
&\le\frac{1}{s}|d_1|_{\infty,S_2(-a)}\intop_{S_2(-a)}\theta^sdS_2\cr
&\le \frac{\varepsilon}{s}\intop_\Omega|\nabla\theta^{s/2}|^2dx+\frac{c}{\varepsilon s}|d_1|_{\infty,S_2(-a)}^2|\theta|_{s,\Omega}^s.\cr}
$$
For sufficiently small $\varepsilon$ we get
$$
\frac{1}{s}\frac{d}{dt}|\theta|_{s,\Omega}^s\le\frac{c}{s}|d_1|_{\infty,S_2(-a)}^2|\theta|_{s,\Omega}^s.
$$
Dropping $s$ and integrating with respect to time yield
$$
|\theta(t)|_{s,\Omega}^s\le\exp\bigg(c\intop_0^t|d_1(t')|_{\infty,S_2(-a)}^2dt'\bigg) |\theta(0)|_{s,\Omega}^s
$$
Using that $|\theta(0)|_{\infty,\Omega}\le\theta^*$ and passing with $s\to\infty$ the above inequality implies
$$
|\theta(t)|_{\infty,\Omega}\le\theta^*.
$$
Next we calculate the lower bound for $\theta$. Multiply $(\ref{2.3})_1$ by $\theta^{-s}$, $s>1$, and integrate over $\Omega$. Then, we have
$$\eqal{
&-\frac{1}{s-1}\frac{d}{dt}\intop_\Omega\frac{1}{\theta^{s-1}}dx-\frac{4s\varkappa}{(s-1)^2} \intop_\Omega\bigg|\nabla{1\over\theta^{(s-1)/2}}\bigg|^2dx\cr
&=\frac{1}{s-1}\intop_\Omega v\cdot\nabla\frac{1}{\theta^{s-1}}dx\cr
&=-\frac{1}{s-1}\intop_{S_2(-a)}d_1\frac{1}{\theta^{s-1}}dS_2+\frac{1}{s-1}\intop_{S_2(a)}d_2\frac{1}{\theta^{s-1}}dS_2.\cr}
$$
Multiplying by $-1$ we get
$$\eqal{
&\frac{1}{s-1}\frac{d}{dt}\intop_\Omega\frac{1}{\theta^{s-1}}dx+{4s\varkappa\over(s-1)^2} \intop_\Omega\bigg|\nabla{1\over\theta^{(s-1)/2}}\bigg|^2dx\cr
&\le\frac{1}{s-1}\intop_{S_2(-a)}d_1\frac{1}{\theta^{s-1}}dS_2\le\frac{1}{s-1}|d_1|_{\infty,S_2(-a)}\intop_{S_2}\frac{1}{\theta^{s-1}}dS_2\cr
&\le\frac{\varepsilon}{s-1}\intop_\Omega\bigg|\nabla\frac{1}{\theta^{(s-1)/2}}\bigg|^2dx+ \frac{c}{\varepsilon(s-1)}|d_1|_{\infty,S_2(-a)}^2\intop_\Omega\frac{1}{\theta^{s-1}}dS_2.\cr}
$$
For sufficiently small $\varepsilon$ we have
$$
\frac{1}{s-1}\frac{d}{dt}\intop_\Omega\frac{1}{\theta^{s-1}}dx\le\frac{c}{s-1}|d_1|_{\infty,S_2(-a)}^2\intop_\Omega\frac{1}{\theta^{s-1}}dx.
$$
Multiplying by $s-1$ and integrating with respect to time yield
$$
\bigg|\frac{1}{\theta}\bigg|_{s-1,\Omega}^{s-1}\le\exp \bigg(c\intop_0^t|d_1(t')|_{\infty,S_2(-a)}^2dt'\bigg)\bigg| \frac{1}{\theta(0)}\bigg|_{s-1,\Omega}^{s-1}.
$$
Passing with $s\to\infty$ gives
$$
{1\over|\theta|_{\infty,\Omega}}\le\frac{1}{|\theta(0)|_{\infty,\Omega}}\quad {\rm so}\quad \theta_*\le|\theta(t)|_{\infty,\Omega}.
$$
Hence (\ref{2.9}) is proved. This concludes the proof.
\end{proof}

Consider the elliptic problem
\begin{equation}\eqal{
&-\nu\divv\D(v)+\nabla p=f\cr
&\divv v=0\cr
&v\cdot\bar n|_{S_1}=0,\ \ \nu\bar n\cdot\D(v)\cdot\bar\tau_\alpha+\gamma v\cdot\bar\tau_\alpha|_{S_1}=0,\ \ \alpha=1,2,\cr
&v\cdot\bar n|_{S_2}=d,\ \ \bar n\cdot\D(v)\cdot\bar\tau_\alpha|_{S_2}=0,\ \ \alpha=1,2.\cr}
\label{2.10}
\end{equation}

\section{Energy estimate for weak solutions}\label{s4}

In this Section we prove the energy estimate for a weak solution to the corrected system where homogeneous Dirichlet boundary conditions on $S_2$, i.e. part of the boundary of domain $\Om$, were imposed. In order to have such solutions we use Hopf function and then a function $\delta$ such that $w=v-\delta$  satisfies $ w\cdot\bar n|_S=0$ so that in equations for $w$ - a weak solution to the system -  it is possible to integrate by parts. It turns out that we need some refined estimates in weighted Sobolev spaces for the function $\delta$ and for solutions to the auxiliary Neumann problem for the Poisson equation - this we have by results in Chapters~{10} and {11} of \cite{RZ2} and also \cite{RZ3}, \cite{RZ4}.

We consider the following system
\begin{equation}\eqal{
&v_t+v\cdot\nabla v-\nu\Delta v+\nabla p=\omega(\theta)f\quad &{\rm in}\ \ \Omega\times\R_+,\cr
&\divv v=0\quad &{\rm in}\ \ \Omega\times\R_+,\cr
&\nu\bar n\cdot\D(v)\cdot\bar\tau_\alpha+\gamma v\cdot\bar\tau_\alpha=0,\ \ \alpha=1,2\quad &{\rm on}\ \ S_1\times\R_+,\cr
&v\cdot\bar n=0\quad &{\rm on}\ \ S_1\times\R_+,\cr
&\bar n\cdot\D(v)\cdot\bar\tau_\alpha=0,\ \ \alpha=1,2,\quad &{\rm on}\ \ S_2\times\R_+,\cr
&v\cdot\bar n=d\quad &{\rm on}\ \ S_2\times\R_+,\cr
&v|_{t=0}=v_0\quad &{\rm in}\ \ \Omega,\cr}
\label{4.1}
\end{equation}
where $\theta$ is a given function.

To find an energy estimate, the integration by parts must be used. To perform this the Dirichlet boundary conditions in $(\ref{4.1})_6$ must be homogeneous. Therefore, following \cite{L1} and \cite{RZ2} we construct a function which makes boundary condition $(\ref{4.1})_6$ homogeneous. For this purpose we use the Hopf function 
\begin{equation}
\eta(\sigma;\varepsilon,\varrho)=\left\{\eqal{
&1\quad & 0\le\sigma\le\varrho e^{-1/\varepsilon}\equiv r,\cr
&-\varepsilon\ln{\sigma\over\varrho}\quad &r<\sigma\le\varrho,\cr
&0\quad &\varrho<\sigma<\infty.\cr}\right.
\label{4.2}
\end{equation}
We calculate
\begin{equation}
{d\eta\over d\sigma}=\eta'(\sigma;\varepsilon,\varrho)=\left\{\eqal{
&0\quad &0<\sigma\le r,\cr
&-{\varepsilon\over\sigma}\quad &r<\sigma\le\varrho,\cr
&0\quad &\varrho<\sigma<\infty,\cr}\right.
\label{4.3}
\end{equation}
so that $|\eta'(\sigma;\varepsilon,\varrho)|\le{\varepsilon\over\sigma}$. We define locally functions $\eta_i$, $i=1,2$, in internal neighborhoods of $S_2$ by setting
$$
\eta_i=\eta_i(\sigma_i;\varepsilon,\varrho),\ \ i=1,2,
$$
where $\sigma_i$ denotes a local coordinate defined on the small neighborhoods
\begin{equation}\eqal{
&S_2(a_1,\varrho)=\{x\in\Omega\colon x_3\in(-a,-a+\varrho)\},\cr
&S_2(a_2,\varrho)=\{x\in\Omega\colon x_3\in(a-\varrho,a)\},\cr
&S_2(\varrho)=S_2(a_1,\varrho)\cup S_2(a_2,\varrho),\cr}
\label{4.4}
\end{equation}
where $\sigma_1=-a+x_3$, $x_3\in(a,a+\varrho),$ $\sigma_2=a-x_3$, $x_3\in(a-\varrho,a)$ and then $\sigma_i$, $i=1,2$, are positive.

Next we set
\begin{equation}
\alpha=\sum_{i=1}^2\tilde d_i\eta_i,\quad b=\alpha\bar e_3,\quad \bar e_3=(0,0,1),
\label{4.5}
\end{equation}
where $\tilde d_i$ is an extension of $d_i$ on $\Omega$, $i=1,2$.

Then we define the function
\begin{equation}
u=v-b.
\label{4.6}
\end{equation}
Therefore
$$\eqal{
&\divv u=-\divv b=-\alpha_{x_3}\quad {\rm in}\ \ \Omega,\cr
&u\cdot\bar n|_S=0.\cr}
$$
After such modifications the boundary conditions for $u$ on $S_2$ are homogeneous but $u$ is not divergence free.

Let us rewrite the compatibility condition as follows
$$
\intop_\Omega\divv b\ dx=\intop_\Omega\alpha_{x_3}dx=-\intop_{S_2(-a)}\alpha|_{x_3=-a}d_1 S_2+ \intop_{S_2(a)}\alpha|_{x_3=a}d_2 S_2=0.
$$
In view of the above compatibility condition, we correct function $u$ defining a function $\varphi$ as a solution to the Neumann problem for the Poisson equation
\begin{equation}\eqal{
&\Delta\varphi=-\divv b\quad &{\rm in}\ \ \Omega,\cr
&\bar n\cdot\nabla\varphi=0\quad &{\rm on}\ \ S,\cr
&\intop_\Omega\varphi dx=0, \cr
& \int_{\Omega} \divv b = 0.\cr}
\label{4.7}
\end{equation}
Consequently, we set
\begin{equation}
w=u-\nabla\varphi=v-(b+\nabla\varphi)\equiv v-\delta.
\label{4.8}
\end{equation}
Therefore, $(w,p, \vartheta)$ is a solution to the problems
\begin{equation}\eqal{
&w_t+(w+\delta)\cdot\nabla w+w\cdot\nabla\delta-\divv\T(w,p)\cr
&=\omega(\vartheta)f-\delta_t-\delta\cdot\nabla\delta+\nu\divv\D(\delta)\equiv \omega(\vartheta)f+ F(\delta,t)\quad &{\rm in}\ \ \Omega^T,\cr
&\divv w=0\quad &{\rm in}\ \ \Omega^t,\cr
&w\cdot\bar n=0\quad &{\rm on}\ \ S^T,\cr
&\nu\bar n\cdot\D(w)\cdot\bar\tau_\alpha+\gamma w\cdot\bar\tau_\alpha=-\nu\bar n\cdot\D(\delta)\cdot\bar\tau_\alpha\cr
&\quad-\gamma\delta\cdot\bar\tau_\alpha\equiv B_{1\alpha}(\delta),\ \ \alpha=1,2\quad &{\rm on}\ \ S_1^T,\cr
&\bar n\cdot\D(w)\cdot\bar\tau_\alpha=-\bar n\cdot\D(\delta)\cdot\bar\tau_\alpha=B_{2\alpha}(\delta),\ \ \alpha=1,2\quad &{\rm on}\ \ S_2^T,\cr
&w|_{t=0}=v(0)-\delta(0)\equiv w(0), \quad &{\rm in}\ \ \Omega,\cr}
\label{4.9}
\end{equation}
\ben \bal \label{2-6}
\vartheta_t-\varkappa\Delta\vartheta+w\cdot\nabla\vartheta & = -\delta\cdot \nabla\vartheta\quad {\rm in}\ \ \Omega^T,\\
{\partial\vartheta\over\partial n} & =  0\quad {\rm on}\ \ S_2^T,\\
{\partial\vartheta\over\partial n} & =  0\quad {\rm on}\ \ S_1^T, \\
\vartheta|_{t=0}& =  \theta_0 =\vartheta_0,\quad {\rm in}\ \ \Omega,
\eal \een
where we used that $\divv\delta=0$. Next, we set
$$\eqal{
&\bar n|_{S_1}= {(\phi_{0,x_1},\phi_{0,x_2},0)\over\sqrt{\phi_{0,x_1}^2+\phi_{0,x_2}^2}}, \quad \bar\tau_1|_{S_1}={(-\phi_{0,x_2},\phi_{0,x_1},0)\over \sqrt{\phi_{0,x_1}^2+\phi_{0,x_2}^2}},\cr
&\bar\tau_2|_{S_1}=(0,0,1),\quad \bar n|_{S_2(-a)}=-\bar e_3,\quad \bar n|_{S_2(a)}=\bar e_3,\cr
&\bar\tau_1|_{S_2}=\bar e_1,\quad \bar\tau_2|_{S_2}=\bar e_2,\cr}
$$
where $\bar e_1=(1,0,0)$, $\bar e_2=(0,1,0)$.

Now, our goal is achieved for the function $w$: the Dirichlet boundary conditions for $w$ are homogeneous and $w$ is divergence free, so we can consider weak solutions to the problem (\ref{4.9}).

Next, introduce the spaces
$$\eqal{
&V_1=\{\psi\in H^1(\Omega):\ \divv\psi=0,\ \psi\cdot\bar n=0\ {\rm on}\ S\},\cr
&V_2=\{\phi\in H^1(\Omega):\ \phi\cdot\bar n=0\ {\rm on}\ S\}.\cr}
$$

\begin{definition}\label{d2-1}\qquad
$1^\circ$ Let $T>0$ be given. We call a function $(w,\vartheta)$ a weak solution to problem (\ref{4.9})--(\ref{2-6}) if $(w,\vartheta)\in(L_2(kT,(k+1)T;V_1)\cap L_\infty(kT,(k+1)T;L_2(\Omega)))\times(L_2(kT,(k+1)T;V_2)\cap L_\infty(kT,(k+1)T;L_2(\Omega)))$ for all $k\in\N_0\equiv\N\cup\{0\}$ and
\begin{equation}\eqal{
&{d\over dt}\intop_\Omega w\cdot\psi dx+\intop_\Omega H(w)\cdot\psi dx+\nu\intop_\Omega\D(w)\cdot\D(\psi)dx\cr
&\quad+\gamma\sum_{\alpha=1}^2\intop_{S_1}w\cdot\bar\tau_\alpha\psi\cdot \bar\tau_\alpha dS-\sum_{\alpha,\beta=1}^2\intop_{S_\beta}B_{\beta\alpha}\psi\cdot \bar\tau_\alpha dS\cr
&=\intop_\Omega \omega(\vartheta)f\cdot\psi dx+\intop_\Omega F(\delta,t)\cdot\psi dx\,\quad \forall\psi\in V_1,\cr}
\label{2-13}
\end{equation}
\begin{equation}\eqal{
&{d\over dt}\intop_\Omega\vartheta\phi dx+\varkappa\intop_\Omega\nabla \vartheta\cdot\nabla\phi dx+\intop_\Omega w\cdot\nabla\vartheta\phi dx\cr
&=-\intop_\Omega\delta\cdot\nabla\vartheta\phi dx dS\quad \forall\phi\in V_2,\cr
&w|_{t=kT}=w(kT),\cr
&\vartheta|_{t=kT}=\vartheta(kT)\cr}
\label{2-14}
\end{equation}
for all $k\in\N_0$, where $H(w)=w\cdot\nabla w+w\cdot\nabla\delta+\delta\cdot\nabla w$.

\noindent
$2^\circ$ We call a function $(v,\theta)$ a weak solution to problem (\ref{1.1}) if $(w,\vartheta)$ is a weak solution to problem (\ref{4.9})--(\ref{2-6}).
\end{definition}

\begin{lemma}\label{l3-1}
Let compatibility condition (\ref{1.11}) hold and $\omega f \in L_{6/5}(\Om).$ Let $d_i\in L_\infty(0,T;W_3^1(S_2(a_i)))\cap L_2(0,T;H^{1/2}(S_2(a_i))), i=1,2,$ $\theta(0)\in L_2(\Omega),$ $\omega f\in L_{6/5}(\Om^T)$,  $d_{it}\in L_2(kT,(k+1)T;W_{6/5}^{1/6}(S_2))$, $i=1,2$, $k\in\N_0$. Then every weak solution of problem (\ref{2-13})--(\ref{2-14}) satisfies for all $t\in[kT,(k+1)T]$, $k\in\N_0$ the estimate
\begin{equation}\eqal{
&{d\over dt}(\tilde c_1\|\vartheta\|_{L_2(\Omega)}^2+\|w\|_{L_2(\Omega)}^2)+\tilde c_2(\tilde c_1\|\vartheta\|_{H^1(\Omega)}^2+\|w\|_{H^1(\Omega)}^2)\cr
&\quad +\gamma\sum_{\alpha=1}^2\|w\cdot\bar\tau_\alpha\|_{L_2(S_1)}^2\le c[\|\om(\te)f\|_{L_{6/5}(\Omega)}^2\cr
&\quad+\psi(\|d\|_{W_{3,\infty(\Omega)}^1})(\|d\|_{W_2^{1/2}(S_2)}^2+\|d_t\|_{W_{6/5}^{1/6}(S_2)}^2)],\cr}
\label{3.3}
\end{equation}
where $\tilde c_i>0$, $i=1,2$, are constants and $\psi$ is a nonlinear increasing function.
\end{lemma}

\begin{remark} \label{r3.3} In the proof below, we have to deal with nonlinear terms of the form $ \de  w \cdot \nb w$ and $\nb \de w \cdot w.$ Therefore, to absorb such terms by $\|w\|_{H^1(\Om^t)}$ on the l.h.s., we need some smallness. Obviously, $w \cdot \nb w,$ $w \cdot w$ are possibly not small, but $L_p$ norms of $\de, \nb \de$ with $ p\ge 3,$  are not small, either. Thus we are going to use weighted norms for components of $\de$ and $ \nb \de,$ presented in Lemma~\ref{l2.2}, for details see  (\ref{3.11}), (\ref{3.13}) and \cite{RZ3}, \cite{RZ3}. 
\end{remark}
\begin{proof}
Setting $\psi=w$ in (\ref{2-13}) and $\phi=\vartheta$ in (\ref{2-14}) as a test function and performing some integrations by parts we will obtain an inequality analogous to (\ref{3.10}) from \cite{RZ2}, i.e. we can derive an energy estimate for weak solutions to (\ref{4.9})-(\ref{2-6}). Namely, we multiply the first equation in (\ref{4.9}) by $\psi$, integrate by parts on $\Om$ and use the definition of $F.$ Then we obtain
\ben \label{3.8} \bal
&\qquad \frac{1}{2}\frac{d}{dt}| w|_{2,\Omega}^2+\int_\Omega ( w\cdot\nb\delta\cdot w+\delta\cdot\nb w\cdot w)dx \\ & \quad - \int_\Omega\div\T( w+\delta,p)\cdot w dx
 =\int_\Omega(\omega(\theta)f-\delta_{t}-\delta\cdot\nb\delta)\cdot w dx.
 \eal \een
 and 
\begin{equation}
{1\over2}{d\over dt}\|\vartheta\|_{L_2(\Omega)}^2+\varkappa \|\nabla\vartheta\|_{L_2(\Omega)}^2=-\intop_\Omega\delta\cdot\nabla\vartheta \vartheta dx.
\label{3-5}
\end{equation}
First, we want to estimate the right-hand side of (\ref{3.8}).
Applying boundary conditions $(\ref{4.9})_{4,5}$ the third integral on the l.h.s. of (\ref{3.8}) can be rewritten as follows
\benn
& & -\int_\Omega\div\T( w+\delta,p)\cdot w dx=-\int_\Omega\div [\nu\D( w+\delta)-p\I]\cdot w dx \\
& & =-\int_\Omega\div[\nu\D( w+\delta)]\cdot w dx+\int_\Omega\nb p \cdot w dx\\
& & =\nu\int_\Omega D_{ij}( w+\delta) w_{j,x_i}dx-\nu\int_\Omega \partial_{x_j}[D_{ij}( w+\delta) w_i]dx+\int_\Omega\div(p w)dx\\
& & \equiv I.
\eenn
and first term in $I$ takes the form
$$
\frac{\nu}{2}\big|D_{ij}(w)\big|_{2,\Omega}^2+\nu\int_\Omega D_{ij}(\delta) w_{j,x_i}dx,
$$
with the summation over repeated indices. By the Green theorem the second term in $I$ reads
\benn
& & -\nu\int_{S_1}n_jD_{ij}( w+\delta) w_idS_1- \nu\int_{S_2}n_jD_{ij} ( w+\delta) w_idS_2\\
& & =-\nu\int_{S_1}n_jD_{ij}( w+\delta)( w_{\tau_\alpha}\tau_{\alpha i}+ w_nn_i)dS_1\\
& & -\nu\int_{S_2}n_jD_{ij}( w+\delta)( w_{\tau_\alpha} \tau_{\alpha i}+ w_nn_i)dS_2\\
& & =\gamma\int_{S_1}(| w_{\tau_\alpha}|^2+ w_{\tau_\alpha} \delta_{\tau_\alpha})dS_1,
\eenn
where $ w_{\tau_\alpha}= w\cdot\bar\tau_\alpha$, $\alpha=1,2$, $ w_n= w\cdot\bar n$.

\noindent
Consequently the last term in $I$ vanishes in view of the Green theorem and $(\ref{4.9})_3$.

With the Korn inequality, we conclude 
\begin{equation}\eqal{
&\frac{1}{2}\frac{d}{dt}|w|_{2,\Omega}^2+ \nu\|w\|_{1,\Omega}^2+\gamma\sum_{\alpha=1}^2|w\cdot\bar\tau_\alpha|_{2,S_1}^2\cr
&\le-\intop_\Omega(w\cdot\nabla\delta\cdot w+\delta\cdot\nabla w\cdot w)dx+c\sum_{\alpha=1}^2|\delta\cdot\bar\tau_\alpha|_{2,S_1}^2\cr
&\quad+c|\D(\delta)|_{2,\Omega}^2+\intop_\Omega(\omega(\theta)f-\delta_t- \delta\cdot\nabla\delta)\cdot wdx.\cr}
\label{3.9}
\end{equation}
The most challenging terms are those caused by nonlinearity $w \cdot \nb w$ and $w \cdot w$ so let us concentrate first on the integral
\ben \label{3.10} \bal
\int_{\Omega}\delta\cdot\nb w\cdot w dx=\int_\Omega (b+\nb\varphi)\cdot\nb w\cdot w dx\\
=\int_{\Omega} b\cdot\nb w\cdot w dx+\int_\Omega\nb\varphi\cdot\nb w \cdot w dx\equiv J_1+J_2. \eal \een
Using the definition of $b$ and H\"older inequality, we estimate $J_1$ by
\benn
|J_1|&\le|\nb w|_{2,\Omega}| w|_{6,\Omega}|b|_{3,\Omega}\le c\| w\|_{1,\Omega}^2|b|_{3,\tilde S_2(\varrho)}\le c\varrho^{1/6} \| w\|_{1,\Omega}^2|b|_{6,\tilde S_2(\varrho)}\\
&\le c\varrho^{1/6}\| w\|_{1,\Omega}^2\|\tilde d\|_{1,\Omega},
\eenn
where $\tilde S_2(\varrho)=S_2(a_1,\varrho)\cup S_2(a_2,\varrho)$.
Next, we estimate $J_2$ in the following way
$$
|J_2|\le|\nb\varphi|_{3,\Omega}| w|_{6,\Omega}|\nb w|_{2,\Omega},
$$
where we estimate $L_3$ norm using weighted spaces as follows
\ben \label{3.11} \bal
|\nb\varphi|_{3,\Omega}&\le c\|\nb\varphi\|_{L_{3,-\mu'}(\Omega)}\le c\|\nb_{x_3}\nb\varphi\|_{L_{3,1-\mu'}(\Omega)}\le c\|\varphi\|_{L_{3,1-\mu'}^2(\Omega)}\\
&\le c\|\div b\|_{L_{3,1-\mu'}(\Omega)},
 \eal \een
with $\mu'>0$. Here the results of Chapters {10} and {11} of \cite{RZ2} (see also \cite{RZ4}) are applied.
Let
\benn
\tilde S'_2(a_1,\varrho) & = & \{x\in\Omega:x_3\in(-a+r,-a+\varrho)\},\\
\tilde S'_2(a_2,\varrho) & = & \{x\in\Omega: x_3\in(a-\varrho,a-r)\}.
\eenn
Using the definition of $b$ and properties of function $\eta,$ for $\mu=1-\mu'$ we obtain the estimates
\ben \bal
\|\div b\|_{L_{3,\mu}(\Omega)}&\le c\varepsilon\bigg(\sum_{i=1}^2
\int_{\tilde S'_2(a_i,\varrho)}|\tilde d_i|^3\frac{\sigma_i^{3\mu}}{\sigma_i^3}dx\bigg)^{1/3}\\
&\quad+\bigg(\sum_{i=1}^2\int_{\tilde S_2(a_i,\varrho)}|\tilde d_{i,x_3}|^3|\sigma_i(x)|^{3\mu}dx\bigg)^{1/3}\\
&\le c\sum_{i=1}^2\varepsilon\bigg(\sup_{x_3}\int_{S_2(a_i,\varrho)}|\tilde d_i|^3dx'\int_r^\varrho \frac{\sigma_i^{2\mu}}{\sigma_i^3} d\sigma_i\bigg)^{1/3}\\
&\quad+\sum_{i=1}^2\bigg(\sup_{x_3}\int_{S_2(a_i,\varrho)}|\tilde d_{i,x_3}|^3dx' \intop_0^\varrho\sigma_i^{3\mu}d\sigma_i\bigg)^{1/3}\\
&\le c\varepsilon\varrho^{\mu-2/3}\sum_{i=1}^2\sup_{x_3}|\tilde d_i|_{3,S_2(a_i)}+c\varrho^{\mu+1/3}\sum_{i=1}^2\sup_{x_3}|\tilde d_{i,x_3}|_{3,S_2(a_i)},
\label{3.12} \eal \een
where $\sigma_i=\dist\{S_2(a_i),x\}$, $x\in\tilde S_2(a_i,\varrho)$.

\noindent
Note that the above inequality is true under the restriction $\mu>\frac{2}{3}$ because for $\mu=\frac{2}{3}$ the first integral on the r.h.s. of the above inequalities is not finite.

Observe that, due to the definition of the weighted spaces used in (\ref{3.11}), the weight is $x_3^{2\mu}$, $3\mu>2$. Therefore, this is not the Muckenhoupt weight, so for the estimate in (\ref{3.11}) of the singular operator see Chapters {10}, {11} of \cite{RZ2} and \cite{RZ3, RZ4}.

Then
$$
|J_2|\le c[\varepsilon\varrho^{\mu-2/3}\sum_{i=1}^2\sup_{x_3}|\tilde d_i|_{3,S_2(a_i)}+ \varrho^{\mu+1/3}\sum_{i=1}^2\sup_{x_3}|\tilde d_{i,x_3}|_{3,S_2(a_i)}]\| w\|_{1,\Omega}^2 \equiv E^2.
$$
Next, we analyze the term
\begin{equation}
\int_\Omega w\cdot\nb\delta\cdot w dx= \int_\Omega w\cdot\nb b\cdot w dx+ \int_\Omega w\cdot\nb\nb\varphi\cdot w dx\equiv J_3+J_4.
\label{3.13}
\end{equation}
For $J_4$, we have
$$
|J_4|=\bigg|\int_\Omega\div( w\cdot\nb\varphi w)dx-\int_\Omega  w\cdot\nb w\cdot\nb\varphi dx\bigg|=\bigg|\int_\Omega w\cdot\nb w\cdot\nb\varphi dx\bigg|\le E^2
$$
because the second term in $J_4$ is equal to $J_2$ and the first vanishes in view of $(\ref{4.9})_3$.

We find also a following bound for $J_3,$ using $b=\alpha\bar e_3=\sum_{i=1}^2\tilde d_i\eta_i\bar e_3$:

\ben \bal
|J_3|&=\bigg|\sum_{i=1}^2\int_{\tilde S_2(a_i,\varrho)} w\cdot\nb (\tilde d_i\eta_i) w_3dx\bigg|\\
&=\bigg|\sum_{i=1}^2\int_{\tilde S_2(a_i,\varrho)}( w\cdot\nb\tilde d_i \eta_i w_3+ w\cdot\nb\eta_i\tilde d_i w_3)dx\bigg|\\
&\le\sum_{i=1}^2\bigg(\int_{\tilde S_2(a_i,\varrho)}| w\cdot\nb\tilde d_i\eta_i|\,| w_3|dx\\
&\quad+\varepsilon\int_{\tilde S_2(a_i,\varrho)} \bigg|\frac{w_3}{\sigma_i} w_3\tilde d_i\bigg|\supp\eta'_id\sigma_idx_1dx_2\bigg)\\
&\le c\sum_{i=1}^2| w|_{6,\tilde S_2(a_i,\varrho)}| w_3|_{3,\tilde S_2(a_i,\varrho)}|\nb\tilde d_i|_{2,\tilde S_2(a_i,\varrho)}\\
&\quad+c\varepsilon\sum_{i=1}^2| w_3|_{6,\tilde S_2(a_i,\varrho)}
|\tilde d_i|_{3,\tilde S_2(a_i,\varrho)}\bigg(\int_{S_2(a_i)}dx_1dx_2 \int_r^\varrho d\sigma_i\bigg|\frac{w_3}{\sigma_i}\bigg|^2\bigg)^{1/2}\\
&\le c\varrho^{1/6}\sum_{i=1}^2| w|_{6,\tilde S_2(a_i,\varrho)}^2
|\nb\tilde d_i|_{2,\tilde S_2(a_i,\varrho)}\\
&\quad+c\varepsilon\sum_{i=1}^2| w|_{6,\tilde S_2(a_i,\varrho)}|\nb w_3|_{2,\tilde S_2(a_i,\varrho)}
|\tilde d_i|_{3,\tilde S_2(a_i,\varrho)}\\
&\le c(\varrho^{1/6}+\varepsilon)\| w\|_{1,\Omega}^2
\|\tilde d\|_{1,3,\Omega}.
\label{3.14} \eal \een
Summarizing the estimates for $J_1$--$J_4$ we infer that the nonlinear term in (\ref{3.9}) can be estimated by
\ben \bal
& &\bigg|\int_\Omega( w\cdot\nb\delta\cdot w+\delta\cdot\nb w \cdot w)dx\bigg|\le c(\varepsilon\varrho^{\mu-2/3}\sum_{i=1}^2\sup_{x_3}|\tilde d_i|_{3,S_2(a_i)}\\
& & \quad+\varrho^{\mu+1/3}\sum_{i=1}^2\sup_{x_3}|\tilde d_{i,x_3}|_{3,S_2(a_i)}+(\varrho^{1/6}+ \varepsilon)\|\tilde d\|_{1,3,\Omega}+\varrho^{1/6}\|\tilde d\|_{1,\Omega}) \| w\|_{1,\Omega}^2.
\label{3.15} \eal \een
Consequently, we analyze the second term on the r.h.s. of (\ref{3.9}),
\benn
\sum_{\alpha=1}^2|\delta\cdot\bar\tau_\alpha|_{2,S_1}^2\le\sum_{\alpha=1}^2 (|b\cdot\bar\tau_\alpha|_{2,S_1}^2+ |\nb\varphi\cdot\bar\tau_\alpha|_{2,S_1}^2)\\
\le|\alpha|_{2,S_1}^2+c\|\nb\varphi\|_{1,3/2,\Omega}^2\le\sum_{i=1}^2 |d_i|_{2,S_1(a_i)}^2+c|\div b|_{3/2,\Omega}^2\\
\le c\|\tilde d\|_{1,3/2,\Omega}^2+c\sum_{i=1}^2
|\nb(\tilde d_i\eta_i)|_{3/2,\Omega}^2\\
\le c\|\tilde d\|_{1,3/2,\Omega}^2+c\sum_{i=1}^2(|\nb\tilde d_i\eta_i|_{3/2,\Omega}^2+|\tilde d_i\nb\eta_i|_{3/2,\Omega}^2)\\
\le c\|\tilde d\|_{1,3/2,\Omega}^2+c\sum_{i=1}^2
|\tilde d_i\nb\eta_i|_{3/2,\Omega}^2.
\eenn
Here, we estimate the last expression above in more detail
\benn
\sum_{i=1}^2|\tilde d_i\nb\eta_i|_{3/2,\Omega}^2\le\varepsilon^2\bigg[\bigg( \intop_{-a+r}^{-a+\varrho}dx_3\int_{S_2(a_1)}dx'
\bigg|\frac{\tilde d_1}{a+x_3}\bigg|^{3/2}\bigg)^{4/3}\\
\quad+\bigg(\int_{a-\varrho}^{a-r}dx_3\int_{S_2(a_2)}dx'
\bigg|\frac{\tilde d_2}{a-x_3}\bigg|^{3/2}\bigg)^{4/3}\bigg]\\
\le\varepsilon^2\bigg[\sup_{x_3}|\tilde d_1|_{3/2,S_2(a_1)}^2\bigg( \intop_{-a+r}^{-a+\varrho}\bigg|\frac{1}{a+x_3}\bigg|^{3/2}dx_3\bigg)^{4/3}\\
\quad+\sup_{x_3}|\tilde d_2|_{3/2,S_2(a_2)}^2\bigg(\int_{a-\varrho}^{a-r} \bigg|\frac{1}{a-x_3}\bigg|^{3/2}dx_3\bigg)^{4/3}\bigg]\\
\le c\varepsilon^2\sum_{i=1}^2 \sup_{x_3}|\tilde d_i|_{3/2,S_2(a_i)}^2\bigg(\int_r^\varrho \frac{dy}{ y^{3/2}}\bigg)^{4/3}\le c\varepsilon^2\sum_{i=1}^2\sup_{x_3}|\tilde d_i|_{3/2,S_2(a_i)}^2 [r^{-1/2}-\varrho^{-1/2}]^{4/3}\\
\le c\varepsilon^2\sum_{i=1}^2\sup_{x_3}|\tilde d_i|_{3/2,S_2(a_i)}^2\frac{1}{\varrho^{2/3}} \left[\exp\left(\frac{1}{2\varepsilon}\right)-1\right]^{4/3}\le c\frac{\varepsilon^2}{\varrho^{2/3}} \exp\left(\frac{2}{3\varepsilon}\right)\sum_{i=1}^2\sup_{x_3}|\tilde d_i|_{3/2,S_2(a_i)}^2.
\eenn
Concluding from the inequalities above, we have
\begin{equation}
\sum_{\alpha=1}^2|\delta\cdot\bar\tau_\alpha|_{2,S_1}^2\le c\|\tilde d\|_{1,3/2,\Omega}^2+c\frac{\varepsilon^2}{\varrho^{2/3}}
\exp\left(\frac{2}{3\varepsilon}\right) \sum_{i=1}^2\sup_{x_3}|\tilde d_i|_{3/2,S_2(a_i)}^2.
\label{3.16}
\end{equation}
Next, we estimate the term
\ben \bal
|\D(\delta)|_{2,\Omega}^2&\le|\D(b)|_{2,\Omega}^2+|\D(\nb\varphi)|_{2,\Omega}^2\\
&\le\sum_{i=1}^2(|\nb\tilde d_i\eta_i|_{2,\Omega}^2+|\tilde d_i\nb\eta_i|_{2,\Omega}^2)+|\nb^2\varphi|_{2,\Omega}^2\\
&\le c\sum_{i=1}^2(|\nb\tilde d_i\eta_i|_{2,\Omega}^2+|\tilde d_i\nb\eta_i|_{2,\Omega}^2)\\
&\le c\sum_{i=1}^2\|\tilde d_i\|_{1,2,\Omega}+c\varepsilon^2\int_{-a+r}^{-a+\varrho}dx_3 \intop_{S_2(a_1)}dx'\bigg|\frac{\tilde d_1}{a+x_3}\bigg|^2\\
&\quad+\varepsilon^2c \int_{a-\varrho}^{a-r}dx_3\int_{S_2(a_2)}dx'\bigg|\frac{\tilde d_2}{a-x_3}\bigg|^2\\
&\le c\sum_{i=1}^2\bigg(\|\tilde d_i\|_{1,2,\Omega}^2+\varepsilon^2\sup_{x_3} |\tilde d_i|_{2,S_2(a_i)}^2\int_r^\varrho\frac{dy}{y^2}\bigg)\\
&\le c\sum_{i=1}^2\bigg[\|\tilde d_i\|_{1,2,\Omega}^2+\varepsilon^2\sup_{x_3}|\tilde d_i|_{2,S_2(a_i)}^2\bigg(\frac{1}{r}-\frac{1}{\varrho}\bigg)\bigg]\\
&\le c\sum_{i=1}^2\left[\|\tilde d_i\|_{1,2,\Omega}^2+\varepsilon^2\sum_{i=1}^2\sup_{x_3}|\tilde d_i|_{2,S_2(a_i)}^2\frac{1}{\varrho}\left(\exp\left(\frac{1}{\varepsilon}\right)-1\right)\right]\\
&\le c\sum_{i=1}^2[\|\tilde d_i\|_{1,2,\Omega}^2+\frac{\varepsilon^2}{\varrho} e^{1/\varepsilon}\sup_{x_3}|\tilde d_i|_{2,S_2(a_i)}^2].
\label{3.17} \eal \een
For the last integral on the r.h.s. of (\ref{3.9}) we infer
\benn
\int_\Omega(\omega(\theta)f-\delta_t-\delta\cdot\nb\delta)\cdot w dx\le\varepsilon_1| w|_{6,\Omega}^2+c(1/\varepsilon_1)(|\omega(\theta)f|_{6/5,\Omega}^2+ |\delta_t|_{6/5,\Omega}^2) \\
+\bigg|\int_\Omega\delta\cdot\nb\delta\cdot  w dx\bigg|.
\eenn
We consider norm of $\de_t$ as follows
\ben \bal
|\delta_t|_{6/5,\Omega}&=|b_t+\nb\varphi_t|_{6/5,\Omega}\le|\tilde d_t|_{6/5,\Omega}+|\div b_t|_{6/5,\Omega}\\
&\le|\tilde d_t|_{6/5,\Omega}+|\nb\tilde d_t|_{6/5,\Omega}+|\tilde d_t\nb\eta|_{6/5,\Omega}\le\|\tilde d_t\|_{1,6/5,\Omega}\\
&\quad+\varepsilon\sum_{i=1}^2 \sup_{x_3}|\tilde d_{i,t}|_{6/5,S_2(a_i)}\bigg(\int_r^\varrho \frac{dx_3}{x_3^{6/5}}\bigg)^{5/6}\\
&\le\|\tilde d_t\|_{1,6/5,\Omega}+c\frac{\varepsilon}{\varrho^{1/6}} e^{1/6\varepsilon}\sum_{i=1}^2\sup_{x_3}|\tilde d_{i,t}|_{6/5,S_2(a_i)},
\label{3.18} \eal \een
where we applied that
$$
\bigg(\int_r^\varrho\frac{dx_3}{x_3^{6/5}}\bigg)^{5/6}=
\bigg(5\frac{1}{x_3^{1/5}}\bigg|_\varrho^r\bigg)^{5/6}=5^{5/6}
\bigg(\frac{1}{r^{1/5}}-\frac{1}{\varrho^{1/5}}\bigg)^{5/6}=
\frac{5^{5/6}}{\varrho^{1/6}}(e^{1/5\varepsilon}-1)^{5/6}.
$$
Finally, we analyze
\ben \bal
& & \bigg|\int_\Omega\delta\cdot\nb\delta\cdot w dx\bigg|\le |\nb\delta|_{2,\Omega}|\delta|_{3,\Omega}| w|_{6,\Omega}\le\varepsilon_2 | w|_{6,\Omega}^2+c(1/\varepsilon_2)\|\delta\|_{1,2,\Omega}^4\\
& & \le\varepsilon_2| w|_{6,\Omega}^2+c(1/\varepsilon_2)(\|\tilde d\|_{1,2,\Omega}^4+\frac{\varepsilon^4}{\varrho^2}e^{2/\varepsilon}\sum_{i=1}^2\sup_{x_3} |\tilde d_i|_{2,S_2(a_i)}^4).
\label{3.19} \eal \een
Summarizing the above estimates in (\ref{3.9}) gives
\ben \bal
&\frac{1}{2}\frac{d}{dt}| w|_{2,\Omega}^2+\nu\| w\|_{1,\Omega}^2+\gamma \sum_{\alpha=1}^2| w\cdot\bar\tau_\alpha|_{2,S_1}^2\\
&\le c\| w\|_{1,\Omega}^2[\varepsilon\varrho^{\mu-2/3}\sum_{i=1}^2\sup_{x_3}|\tilde d_i|_{3,S_2(a_i)}+\varrho^{\mu+1/3}\sum_{i=1}^2\sup_{x_3}|\tilde d_{i,x_3}|_{3,S_2(a_i)}\\
&\quad+(\varrho^{1/6}+\varepsilon)\|\tilde d\|_{1,3,\Omega}]\\
&\quad+c[|\omega(\theta)f|_{6/5,\Omega}^2+\|\tilde d\|_{1,\Omega}^4+\|\tilde d\|_{1,\Omega}^2+ \|\tilde d_t\|_{1,6/5,\Omega}^2\\
&\quad+\frac{\varepsilon^2}{\varrho}e^{1/\varepsilon}\sum_{i=1}^2\sup_{x_3}|\tilde d_i|_{2,S_2(a_i)}^2+\frac{\varepsilon^4}{\varrho^2}e^{2/\varepsilon}\sum_{i=1}^2\sup_{x_3}|\tilde d_i|_{2,S_2(a_i)}^4\\
&\quad+\frac{\varepsilon^2}{\varrho^{2/3}}e^{2/3\varepsilon}\sum_{i=1}^2\sup_{x_3} |\tilde d_i|_{3/2,S_2(a_i)}^2+\frac{\varepsilon^2}{\varrho^{1/3}}e^{1/3\varepsilon}\sup_t|\tilde d_t|_{6/5,S_2}^2].
\label{3.20} \eal \een
The norms of $\tilde d$ under the first square bracket on the r.h.s. of (\ref{3.20}) are bounded by
$$
\|\tilde d\|_{W_{3,\infty}^1(\Omega)}
$$
and the first three norms of $\tilde d$ under the second square bracket are bounded by the quantities
$$
\|\tilde d\|_{1,\Omega},\quad \|\tilde d_t\|_{1,6/5,\Omega}.
$$
with the following norms of anisotropic Sobolev spaces
\benn
\|u\|_{W_{p_1,p_2}^1(\Omega)}=\bigg(\int_{-a}^a\bigg(\int_{S_2} \sum_{|\alpha|\le1}|D_x^\alpha u|^{p_1}dx_1dx_2\bigg)^{p_2/p_1}dx_3\bigg)^{1/p_2},\\
\|u\|_{L_{p_1,p_2}(\Omega)}=\bigg[\bigg(\int_{S_2}|u|^{p_1}dx_1dx_2 \bigg)^{p_2/p_1}dx_3\bigg]^{1/p_2},
\eenn
where $p_1,p_2\in[1,\infty]$.

Then, we reformulate (\ref{3.20}) as follows
\ben \bal
&\frac{1}{2}\frac{d}{dt}| w|_{2,\Omega}^2+\nu\| w\|_{1,\Omega}^2+\gamma \sum_{\alpha=1}^2| w\cdot\bar\tau_\alpha|_{2,S_1}^2\\
&\le c\| w\|_{1,\Omega}^2[(\varepsilon\varrho^{\mu-2/3}+\varrho^{\mu+1/3}+ \varrho^{1/6}+\varepsilon)\|\tilde d\|_{W_{3,\infty}^1(\Om)}]\\
&\quad+c\bigg[|\omega(\te)f|_{6/5,\Omega}^2+\|\tilde d\|_{1,\Omega}^2+\|\tilde d\|_{1,\Omega}^4+ \|\tilde d_t\|_{1,6/5,\Omega}^2\\
&\quad+\bigg(\frac{\varepsilon^2}{\varrho}e^{1/\varepsilon}+ \frac{\varepsilon^2}{\varrho^{2/3}}e^{2/3\varepsilon}\bigg)\|\tilde d\|_{L_{2,\infty}(\Omega)}^2+\frac{\varepsilon^4}{\varrho^2}e^{2/\varepsilon} \|\tilde d\|_{L_{2,\infty}(\Omega)}^4\\
&\quad+\frac{\varepsilon^2}{\varrho^{1/3}}e^{1/3\varepsilon}\|\tilde d_t\|_{L_{6/5,\infty}(\Omega)}^2\bigg].
\label{3.21} \eal \een
We demand on parameters that $\mu>\frac{2}{3}$ and $\varrho<1$ so that $\varrho^{\mu+1/3}<\varrho^{1/6}$.
Consequently, the expression under the first square bracket in (\ref{3.21}) can be estimated by
$$
(\varepsilon+\varrho^{1/6})\|\tilde d\|_{W_{3,\infty}^1(\Omega)}.
$$
We set
\begin{equation}
\varepsilon=\frac{\nu}{8c\|\tilde d\|_{W_{3,\infty}^1(\Omega)}},\quad \varrho=\bigg(\frac{\nu}{8c\|\tilde d\|_{W_{3,\infty}^1(\Omega)}}\bigg)^6.
\label{3.24}
\end{equation}
Then, it follows
\begin{equation}
c(\varepsilon+\varrho^{1/6})\|\tilde d\|_{W_{3,\infty}^1(\Omega)}\le\frac{\nu}{4}
\label{3.22}
\end{equation}
and infer from (\ref{3.21}) the relation
\ben \bal
&\frac{d}{dt}| w|_{2,\Omega}^2+\nu\| w\|_{1,\Omega}^2+2\gamma \sum_{\alpha=1}^2| w\cdot\bar\tau_\alpha|_{2,S_1}^2\\
&\le c[|\omega(\te)f|_{6/5,\Omega}^2+\|\tilde d\|_{1,\Omega}^2+\|\tilde d\|_{1,\Omega}^4+ \|\tilde d_t\|_{1,6/5,\Omega}^2]\\
&\quad+c\bigg[\bigg(\frac{\varepsilon^2}{\varrho}e^{1/\varepsilon}+ \frac{\varepsilon^2}{\varrho^{2/3}}e^{2/3\varepsilon}\bigg)
\|\tilde d\|_{L_{2,\infty}(\Omega)}^2\\
&\quad+\frac{\varepsilon^4}{\varrho^2}e^{2/\varepsilon}
\|\tilde d\|_{L_{2,\infty}(\Omega)}^4+\frac{\varepsilon^2}{\varrho^{1/3}}e^{1/3} \|\tilde d_t\|_{L_{6/5,\infty,\Omega}}^2\bigg].
\label{3-23} \eal \een

Then (\ref{3-23}) can be reformulated as follows 
\ben \bal
& & \frac{d}{dt}| w|_{2,\Omega}^2+\nu\| w\|_{1,\Omega}^2+2\gamma\sum_{\alpha=1}^2 | w\cdot\bar\tau_\alpha|_{2,S_1}^2\\
& & \le c[|\omega(\te)f|_{6/5,\Omega}^2+\|\tilde d\|_{1,\Omega}^2+\|\tilde d\|_{1,\Omega}^4+ \|\tilde d_t\|_{1,6/5,\Omega}^2]\\
& &\quad+c\bigg[(1+\|\tilde d\|_{W_{3,\infty}^1(\Omega)}^2)
\|\tilde d\|_{W_{3,\infty}^1(\Omega)}^4\exp\bigg(\frac{1}{c}
\|\tilde d\|_{W_{3,\infty}^1(\Omega)}^2\bigg)\\
& & \quad+\|\tilde d\|_{W_{3,\infty}^1(\Omega)}^{10}\exp\bigg(\frac{1}{c}
\|\tilde d\|_{W_{3,\infty}^1(\Omega)}\bigg)+\exp\bigg(\frac{1}{c}
\|\tilde d\|_{W_{3,\infty}^1(\Omega)}\bigg)
\|\tilde d_t\|_{L_{6/5,\infty,\Omega}}^2\bigg] \\
& & \le \bar c_1[\|\om(\te)f\|_{L_{6/5}(\Omega)}^2+\|\vartheta\|_{L_2(\Omega)}^2 \|\om(\te)f\|_{L_3(\Omega)}^2\\
& & \quad+\psi(\|\bar d\|_{W_{3,\infty}^1(\Omega)})(\|\bar d\|_{H^1(\Omega)}^2+\|\bar d_t\|_{L_{6/5,\infty}(\Omega)}^2)],
\label{3-25}  \eal \een
Now, we consider the terms on the right-hand side of the heat equation (\ref{3-5}). Using the H\"older and Poincar\'e inequalities we get
\begin{equation}\eqal{
&\bigg|\intop_\Omega\delta\cdot\nabla\vartheta\vartheta dx\bigg|\le \|\delta\|_{L_3(\Omega)}\|\nabla\vartheta\|_{L_2(\Omega)} \|\vartheta\|_{L_6(\Omega)}\cr
&\le c\|\delta\|_{L_3(\Omega)}\|\nabla\vartheta\|_{L_2(\Omega)}^2\le c(\|b\|_{L_3(\Omega)}+\|\nabla\phi\|_{L_3(\Omega)}) \|\nabla\vartheta\|_{L_2(\Omega)}^2.\cr}
\label{3-15}
\end{equation}
To consider the norm $\|\delta\|_{L_3(\Omega)}$ we use that $\delta=b+\nabla\phi$, where $b$ is given by (\ref{2.1})--(\ref{2.2}) and $\phi$ is a solution of (\ref{2.3}). First, we use the H\"older inequality
$$
\|b\|_{L_3(\Omega)}^3\le c\varrho^{1/2}\sum_{i=1}^2\|\bar d_i\|_{L_6(\Omega)}^3.
$$
Next, by (\ref{3.11}) we have
\begin{equation}
\|\nabla\phi\|_{L_3(\Omega)}\le c\|\divv b\|_{L_{3,1-\mu'}(\Omega)},
\label{3-16}
\end{equation}
where ${2\over3}<1-\mu'\le1$.

\noindent
Let $\mu=1-\mu'$. Then (\ref{3.12}) yields
\begin{equation}\eqal{
&\|\divv b\|_{L_{3,\mu}(\Omega)}\cr
&\le c(\varepsilon\varrho^{\mu-{2\over3}}\sup_{x_3}\|\bar d\|_{L_3(S_2)}+ \varrho^{\mu+{1\over3}}\sup_{x_3}\|\bar d_{x_3}\|_{L_3(S_2)}).\cr}
\label{3-17}
\end{equation}
Using (\ref{3-15}) in (\ref{3-5}), and assuming that $\varepsilon_4$ is sufficiently small we obtain
\begin{equation}\eqal{
&{1\over2}{d\over dt}\|\vartheta\|_{L_2(\Omega)}^2+{\varkappa\over2} \|\nabla\vartheta\|_{L_2(\Omega)}^2\le c\|\nabla\vartheta\|_{L_2(\Omega)}^2\cr
&\quad\cdot(\varrho^{1/6}\|\bar d\|_{L_6(\Omega)}+\varepsilon\varrho^{\mu-{2\over3}}\sup_{x_3}\|\bar d\|_{L_2(S_2)}+\varrho^{\mu+{1\over3}}\sup_{x_3}\|\bar d_{x_3}\|_{L_3(S_2)}).\cr}
\label{3-20}
\end{equation}
Assume that $\varrho$, $\varepsilon$ are so small that
$$
c(\varrho^{1/6}\|\bar d\|_{L_6(\Omega)}+\varepsilon\varrho^{\mu-{2\over3}}\sup_{x_3}\|\bar d\|_{L_2(S_2)}+\varrho^{\mu+{1\over3}}\sup_{x_3}\|\bar d_{x_3}\|_{L_3(S_2)})
\le{\varkappa\over 4}.
$$
Then inequality (\ref{3-20}) implies
\begin{equation}
{d\over dt}\|\vartheta\|_{L_2(\Omega)}^2+{\bar c_2\varkappa\over2} \|\vartheta\|_{H^1(\Omega)}^2\le 0,
\label{3-21}
\end{equation}
where the Poincar\'e inequality has been also used, $\bar c_2>0$ is constant.
Now, multiply (\ref{3.21}) by a sufficiently large constant $\bar c_4>0$ which will be chosen later and add the result to the inequality (\ref{3.14}).

\noindent
We get
\begin{equation}\eqal{
&{d\over dt}(\|w\|_{L_2(\Omega)}^2+\bar c_4\|\vartheta\|_{L_2(\Omega)}^2)+ {\nu\over2}\|w\|_{H^1(\Omega)}^2+{\bar c_2\bar c_4\varkappa\over2}\|\vartheta\|_{H^1(\Omega)}^2\cr
&\quad+\gamma \sum_{\alpha=1}^2\|w\cdot\bar\tau_\alpha\|_{L_2(S_1)}^2\cr
&\le\bar c_1[\|\om(\te)f\|_{L_{6/5}(\Omega)}^2+\|\vartheta\|_{L_2(\Omega)}^2 \|\om(\te)f\|_{L_3(\Omega)}^2 \cr  & \quad+\psi(\|\bar d\|_{W_{3,\infty}^1(\Omega)})(\|\bar d\|_{H^1(\Omega)}^2+\|\bar d_t\|_{L_{6/5,\infty}(\Omega)}^2)].\cr}
\label{3-22}
\end{equation}
Let $
{\rm ess}\sup_{t\in (0,T)}\|\om(\te)f\|_{L_{6/5}(\Omega)}^2=M$. Assume that $\bar c_4$ is so large that \\ ${\bar c_2\bar c_4\varkappa\over4}=\bar c_1M$.
Then inequality (\ref{3-22}) yields
$$\eqal{
&{d\over dt}(\|w\|_{L_2(\Omega)}^2+\tilde c_1\|\vartheta\|_{L_2(\Omega)}^2)+\tilde c_2(\|w\|_{H^1(\Omega)}^2
+\tilde c_1\|\vartheta\|_{H^1(\Omega)}^2) \cr &\quad +\gamma\sum_{\alpha=1}^2 \|w\cdot\bar\tau_\alpha\|_{L_2(S_1)}^2\cr
&\le c[\|\om(\te)f\|_{L_{6/5}(\Omega)}^2+\psi(\|\bar d\|_{W_{3,\infty}^1(\Omega)})\cdot(\|\bar d\|_{H^1(\Omega)}^2+\|\bar d_t\|_{L_{6/5,\infty}(\Omega)}^2)],\cr}
$$
where $\tilde c_1=\tilde c_4$, $\tilde c_2=\min\big({\nu\over2},{\bar c_2\over2}\varkappa\big)$. Next, we apply the Sobolev anisotropic imbeddings (see \cite[Ch. 3, Sect. 10]{BIN})
$$
\|\bar d_t\|_{L_{6/5,\infty(\Omega)}}\le c\|\bar d\|_{W_{6/5(\Omega)}^1}.
$$
This ends the proof.
\end{proof}
Now, introduce the following notation:
$$\eqal{
&X(t)=\tilde c_1\|\vartheta(t)\|_{L_2(\Omega)}^2+\|w(t)\|_{L_2(\Omega)}^2,\cr
&Y(t)=\tilde c_1\|\vartheta(t)\|_{H^1(\Omega)}^2+\|w(t)\|_{H^1(\Omega)}^2,\cr
&A_1(T)=\sup_{k\in\N_0}\intop_{kT}^{(k+1)T}F(t)dt,\cr}
$$
where the function $F=F(t)$ denotes the right-hand side of (\ref{3.3}).

\begin{lemma}\label{l3.2}
Let the assumptions of Lemma \ref{l3-1} hold. Assume that\\ $A_1(T)<\infty$ and $X(0)<\infty$. Then
\begin{equation}\eqal{
&{\rm ess} \sup_{kT\le\tau\le T}X(\tau)+\tilde c_2\intop_{kT}^tY(\tau)d\tau\le A_1(T){2-e^{-\tilde c_2T}\over1-e^{-\tau_2T}}+X(0)\equiv A_2(T)\cr}
\label{3.23}
\end{equation}
for all $t\in(kT,(k+1)T)$, $k\in\N_0$, where $\tilde c_2>0$ is the constant from (\ref{3.3}).
\end{lemma}

\begin{proof}
Using the above notation inequality (\ref{3.3}) takes the form
$$
{dX\over dt}+\tilde c_2Y\le F,
$$
which implies
\begin{equation}
{d\over dt}(X(t)e^{\tilde c_2t})\le F(t)e^{\tilde c_2t}.
\label{3.25}
\end{equation}
Integrating (\ref{3.25}) with respect to time from $t=kT$ to $t\in(kT,(k+1)T)$, $k\in\N_0$, we obtain
$$
X(t)\le e^{-\tilde c_2t}\intop_{kT}^tF(\tau)e^{\tilde c_2\tau}d\tau+e^{-\tilde c_2(t-kT)}X(kT).
$$
Let $t=(k+1)T$. Then
$$
X((k+1)T)\le\intop_{kT}^{(k+1)T}F(t)dt+e^{-\tilde c_2T}X(kT).
$$
By iteration we get
\begin{equation}
X(kT)\le{A_1(T)\over1-e^{-\tilde c_2T}}+e^{-\tilde c_2kT}X(0).
\label{3.26}
\end{equation}
Now, integrating (\ref{3.24}) with respect to time from\\  $t=kT$ to $t\in(kT,(k+1)T)$ and using (\ref{3.26}) yields
$$\eqal{
&{\rm ess} \sup_{kT\in\tau\le t}X(\tau)+\tilde c_2\intop_{kT}^tY(\tau)d\tau\le\intop_{kT}^tF(\tau)d\tau+X(kT)\cr
&\le A_1(T)+{A_1(T)\over1-e^{-\tilde c_2T}}+e^{-\tilde c_2kT}X(0)\le A_2(T).\cr}
$$
\end{proof}

\begin{lemma} Assume that $(v,p,\theta)$ is a solution to (\ref{1.1}) and $w$ is a weak solution to problem (\ref{4.9}). Let $d_i\in L_\infty(0,T;W_3^1(S_2(a_i)))\cap L_2(0,T;H^{1/2}(S_2(a_i))),$ $i=1,2.$ Then 
\begin{equation}
\|v\|_{V(\Omega^t)}^2\le\| w\|_{V(\Omega^t)}^2+ \|\delta\|_{V(\Omega^t)}^2,
\label{3.34}
\end{equation} where  
\ben \label{3.39} \bal
\|\delta\|_{V(\Omega^t)}^2 \le c|\tilde d|_{2,\infty,\Omega^t}^2+c\|\tilde d\|_{1,2,2,\Omega^t}^2\\
+c\sup_t\|\tilde d\|_{W_{3,\infty}^1(\Omega)}^4\exp\left(
\frac{1}{c}\sup_t\|\tilde d\|_{W_{3,\infty}^1(\Omega)}\right)\\ \cdot\int_0^t\sum_{i=1}^2\sup_{x_3} |\tilde d_i|_{2,S_2(a_i)}^2dt'
+c\sup_{t'\le t}\|d(t')\|_{L_{2,\infty}(\Omega)}^2
 \eal \een and $\tilde d_i$ is an extension of $d_i$, $i=1,2$ on $\Om.$ \end{lemma}
\begin{proof}
Having estimate for $ w$ and using transformation (\ref{4.8}) we will derive
(\ref{3.34}).
\noindent
The last term in the r.h.s. of (\ref{3.34}) can be treated as
\ben \label{3.35} \bal
\|\delta\|_{V(\Omega^t)}^2 \le\|b\|_{V(\Omega^t)}^2+ \|\nb\varphi\|_{V(\Omega^t)}^2\le|b|_{2,\infty,\Omega^t}+ \|b\|_{1,2,2,\Omega^t}^2\\
+|\nb\varphi|_{2,\infty,\Omega^t}^2+ \|\nb\varphi\|_{1,2,2,\Omega^t}^2\equiv I.
 \eal \een
Now, we analyze the particular terms in $I$.

\noindent
The first term,
$$
|b|_{2,\infty,\Omega^t}^2\le|\alpha|_{2,\infty,\Omega^t}^2\le\sum_{i=1}^2 |\tilde d_i|_{2,\infty,\Omega^t}^2=|\tilde d|_{2,\infty,\Omega^t}.
$$
For the second term we have
\benn
\|b\|_{1,2,2,\Omega^t}^2=\int_0^t(|\nb b|_{2,\Omega}^2+ |b|_{2,\Omega}^2)dt'\le\sum_{i=1}^2\int_0^t|\nb\tilde d_i\eta_i|_{2,\Omega}^2dt'\\
\quad+\sum_{i=1}^2 \int_0^t|\tilde d_i\nb\eta_i|_{2,\Omega}^2dt'+\sum_{i=1}^2\int_0^t|\tilde d_i|_{2,\Omega}^2dt'\\
\le c\|\tilde d\|_{1,2,2,\Omega^t}^2+\varepsilon^2\bigg[\int_0^t \intop_{-a+r}^{-a+\varrho}dx_3\int_{S_2(a_1)}dx'\bigg|\frac{\tilde d_1}{ a+x_3}\bigg|^2\\
\quad+\int_0^tdt'\int_{a-\varrho}^{a-r}dx_3\int_{S_2(a_2)}dx'\bigg| \frac{\tilde d_2}{ a-x_3}\bigg|^2\bigg]\\
\le c\|\tilde d\|_{1,2,2,\Omega^t}^2+\varepsilon^2\int_0^t\bigg[\sup_{x_3} |\tilde d_1|_{2,S_2(a_1)}^2\int_{-a+r}^{-a+\varrho}\bigg|\frac{1}{ a+x_3}\bigg|^2dx_3\\
\quad+\sup_{x_3}|\tilde d_2|_{2,S_2(a_2)}^2\int_{a-\varrho}^{a-r}\bigg| \frac{1}{ a-x_3}\bigg|^2dx_3\bigg]dt'\\
\le c\|\tilde d\|_{1,2,2,\Omega^t}^2+c\varepsilon^2\int_0^t\sum_{i=1}^2\sup_{x_3}|\tilde d_i|_{2,S_2(a_i)}^2dt' \int_r^\varrho\frac{dy}{y^2}\\
=c\|\tilde d\|_{1,2,2,\Omega^t}^2+c\varepsilon^2\int_0^t\sum_{i=1}^2\sup_{x_3}
|\tilde d_i|_{2,S_2(a_i)}^2dt'\bigg(\frac{1}{r}-\frac{1}{\varrho}\bigg)
\eenn
\benn
=c\|\tilde d\|_{1,2,2,\Omega^t}^2+c\varepsilon^2\int_0^t\sum_{i=1}^2\sup_{x_3}
|\tilde d_i|_{2,S_2(a_i)}^2dt'\frac{1}{\varrho}\left[\exp\left(\frac{1}{\varepsilon}\right)-1\right]\\
\le c\|\tilde d\|_{1,2,2,\Omega^t}+c\frac{\varepsilon^2}{\varrho}\exp \bigg(\frac{1}{\varrho}\bigg)\int_0^t\sum_{i=1}^2\sup_{x_3}|\tilde d_i|_{2,S_2(a_i)}^2dt'\equiv J_1.
\eenn
Applying parameters settings (\ref{3.24}) gives
\ben \label{3.36} \bal
J_1&\le c\|\tilde d\|_{1,2,2,\Omega^t}^2+c\sup_t\|\tilde d\|_{W_{3,\infty}^1(\Omega)}\exp\bigg(\frac{\sup_t\|\tilde d\|_{W_{3,\infty}^1(\Omega)}}{ c}\bigg)\\
&\quad\cdot\int_0^t\sum_{i=1}^2\sup_{x_3}|\tilde d_i|_{2,S_2(a_i)}^2dt'\equiv J.
 \eal \een
Next, we analyze the third term in $I$ defined in (\ref{3.35}).
$$
|\nb\varphi|_{2,\infty,\Omega^t}^2=\sup_t|\nb\varphi(t')|_{2,\Omega}^2
$$
We multiply $(\ref{4.7})_1$ by $\varphi$, integrate by parts and apply $(\ref{4.7})_2$ to get
\ben \label{3.37} \bal
|\nb\varphi|_{2,\Omega}^2&=\int_\Omega\div b\,\varphi dx=\int_\Omega\div (b\varphi)dx-\int_\Omega b\cdot\nb\varphi dx\\
&=\sum_{i=1}^2\int_{S_2(a_i)}d_i\varphi dS_2-\int_\Omega b\cdot\nb\varphi dx.
 \eal \een
Using the H\"older and Young inequalities to the r.h.s. of (\ref{3.37}) and employing the Poincar\'e inequality to the l.h.s. of (\ref{3.37}) we have
$$
\|\varphi\|_{1,\Omega}^2\le\varepsilon_1|\nb\varphi|_{2,\Omega}^2+ \frac{c}{\varepsilon_1}|b|_{2,\Omega}^2+\varepsilon_2|\varphi|_{2,S_2}^2+ \frac{c}{\varepsilon_2}\sum_{i=1}^2|\tilde d_i|_{2,S_2(a_i)}^2.
$$
Then, for sufficiently small $\varepsilon_1$ and $\varepsilon_2$ we derive
$$
\|\varphi\|_{1,\Omega}^2\le c|b|_{2,\Omega}^2+c\sum_{i=1}^2|d_i|_{2,S_2(a_i)}^2\le c|d|_{L_{2,\infty}(\Omega)}^2
$$
Further, we get
\begin{equation}
\|\varphi\|_{1,2,\infty,\Omega^t}\le c\sup_{t'\le t}|d(t')|_{L_{2,\infty}(\Omega)}.
\label{3.38}
\end{equation}
Finally, we consider the last term in $I$ 
$$
\|\nb\varphi\|_{1,2,2,\Omega^t}^2=\int_0^t(|\nb^2\varphi|_{2,\Omega}^2+ |\nb\varphi|_{2,\Omega}^2)dt'\equiv J_2.
$$
Applying solvability of problem (\ref{4.7}) we obtain
$$
J_2\le c\int_0^t|d(t')|_{L_{2,\infty}(\Omega)}^2dt'+c\int_0^t|\div b(t')|_{2,\Omega}^2dt'\equiv J_3.
$$
The second term in $J_3$ can be estimated by
$$
c\int_0^t\sum_{i=1}^2(|\nb\tilde d_i|_{2,\Omega}^2+|\tilde d_i\nb\eta_i|_{2,\Omega}^2)dt'\le cJ,
$$
with $J$ defined in $(\ref{3.36})$. Summarizing, we conclude estimate (\ref{3.39}) for $\de$
and this concludes the proof. \end{proof}
\noindent
The above considerations imply the following result:

\begin{theorem}\label{l3.1}
Let $(v,p,\theta)$ be a solutions to problem (\ref{1.1}) and $ w$ be defined by (\ref{4.9}). Assume that $d_i\in L_\infty(0,T;W_3^1(S_2(a_i)))\cap L_2(0,T;H^{1/2}(S_2(a_i))),$  $\omega f\in L_{6/5}(\Om^T),$ $\tilde d_i$ is an extension of $d_i$, $i=1,2.$ Then
\begin{itemize}
\item[(1)]
$$
\| w\|_{V(\Om^t)}^2\le cA_1^2(T)+\|w(0)\|_{L_2(\Omega)}^2,
$$
where $T>0$, $t\in[0,T]$ and
\benn
A_1^2(T)=\int_{0}^{T}(\|\omega(\theta)f(t)\|_{L_{6/5}(\Omega)}^2+ \sum_{i=1}^2\|d_i(t)\|_{W^1_3(S_2(a_i))}^2 \\ +\sum_{i=1}^2\|d_{i,t}(t)\|_{W^1_{6/5}(S_2(a_i))}^2)dt \cdot
\phi(\sum_{i=1}^2\sup_t\|d_i(t)\|_{W^1_3(S_2(a_i))})<\infty,
\eenn
where $\phi$ is an increasing positive function.
Next
\item[(2)]
\ben \bal
& & \|v\|_{V(\Omega^T)}^2 + \|\theta\|_{V(\Omega^T)}^2 \le  A_1(T)^2+\|w(0)\|_{L_2(\Omega)}^2 +\|\vartheta(0)\|_{L_2(\Omega)}^2  \\ & & +c\sup_{t'\in(0,T)}\|\tilde d(t')\|_{L_2(\Omega)}^2+c\intop_{0}^T\|\tilde d(t')\|_{W^1_2(\Omega)}^2dt' \\ & & \qquad
+c\sup_{t'\in(0,T]}\|\tilde d(t')\|_{W_{3,\infty}^1(\Omega)}^4\exp\bigg( {1\over c}\sup_{t'\in(0,T)}\|\tilde d(t')\|_{W_{3,\infty}^1(\Omega)}\bigg)\cdot \\
& & \qquad\cdot\intop_{0}^T\sum_{i=1}^2\sup_{x_3}\|\tilde d_i(t')\|_{L_2(S_2(a_i))}^2dt' +c\|\om(\te)f\|_{L_{6/5}(\Omega)}^2 \equiv A^2(T).
\label{4.26} \eal
\een
 
\end{itemize}
\end{theorem}

\begin{proof}
Property (1) follows from (\ref{3.3}) and property (2) from (\ref{3.34}) and (\ref{3.39}). This ends the proof.
\end{proof}

\begin{theorem}(energy estimate for temperature) \label{l5.8}
Let $\theta(0)\in L_2(\Omega),$ $d_1\in L_6(0,t;L_3(S_2(-a))) \cap L_2(0,t; L_{\infty}(S_2(-a)).$ Then $\theta$ is bounded from above with positive $\theta^*$ and solutions to (\ref{2.3}) satisfy
\ben \bal
&\|\theta\|^2_{V(\Om^t)} \le c\exp(\|d_1\|_{L_6(0,t;L_3(S_2(-a))}^6)\|\theta(0)\|_{L_2(\Omega)}^2 
\\ & \quad + \|d_1\|^2_{L_1(S_2^t(-a))}{\theta^*}^2+ \|\theta(0)\|_{L_1(\Om)}^2.
\label{5.30}
\eal \een
\end{theorem}

\begin{proof}
Multiply $(\ref{2.3})_1$ by $\theta$ and integrate over $\Omega$. Then it follows
\ben \bal 
\frac{1}{2}\frac{d}{dt}|\theta|_{2,\Omega}^2+\varkappa\intop_\Omega|\nabla\theta|^2dx=-\frac{1}{2}\intop_\Omega v\cdot\nabla\theta^2dx = -\frac{1}{2} \intop_{\Omega} \divv (v\theta^2) dx \\ \le \frac{1}{2} \int_{S_2(-a)} \tilde{d}_1 \theta^2d S_2 - \frac{1}{2} \int_{S_2(a)} \tilde{d}_2 \theta^2 d S_2\le \frac{1}{2} |d_1|_{3,S_2(-a)} |\theta|_{3,S_2(-a)}^2 \\ \le \frac{\varkappa}{2}|\nb \theta|_{2,\Omega}^2 + \frac{c}{\varkappa} |d_1|_{3,S_2(-a)}^6 |\theta|_{2,\Om}^2.
\label{5.31}
\eal \een
Hence 
\benn 
\frac{d}{dt} |\theta|_{2,\Omega}^2+\varkappa |\nb\theta|_{2,\Omega}^2 \le c |d_1|_{3,S_2(-a)}^6 |\theta(0)|_{2,\Omega}^2 
\eenn
Continuing,
\benn 
\frac{d}{dt} (|\theta|_{2,\Omega}^2 e^{-c|d_1|_{3,6,S_2^t(-a)}^6 })+ \varkappa|\nb\theta|^2_{2,\Omega}e^{-c|d_1|_{3,6,S_2^t(-a)}^6 } \le 0 \eenn so that
\ben \label{1} \bal |\theta|_{2,\Omega}^2 +\varkappa |\nb\theta|_{2,\Omega}^2 \le e^{c|d_1|_{3,6,S_2(-a)}^6 } |\theta(0)|_{2,\Omega}^2. 
\eal \een
In order to estimate the norm 
\benn \|\theta\|_{V(\Omega^t)} = {\rm ess}
\sup_{t'\in(0,t)}|\theta(t)|_{2,\Omega} + \|\theta\|_{L_2(0,t;H^1(\Omega))} \eenn
we conclude from $(\ref{2.3})_1$ 
\benn \frac{d}{dt} \int_{\Om} \theta dx = -\int_{\Om} v\cdot \nb \theta dx \le \int_{S_2(-a)} \tilde{d}_1 \theta dS_2 \le |d_1|_{1,S_2(-a)} \theta^* \eenn
Hence \benn \int_{\Omega} \te dx \le |d_1|_{1,S_2^t(-a)} \theta^* + \int_{\Omega} \theta(0) dx \eenn
and 
\ben \bal \label{2} |\theta|_{2,\Om} = \left|\theta - \frac{1}{|\Om|}\int_{\Om} \te dx + \frac{1}{|\Om|}\int_{\Om} \te dx\right|_{2,\Om} \\ \le \left|\theta - \frac{1}{|\Om|}\int_{\Om} \te dx\right|_{2,\Om} + \frac{1}{|\Om|}\left|\int_{\Om} \te dx\right|_{2,\Om} \le c |\nb \te|_{2,\Om} +  |\int_{\Om} \te dx| \eal \een
From (\ref{1}), (\ref{2}) 
\benn \|\te\|^2_{V(\Om^t)} \le c e^{c|d_1|_{3,6,S_2^t(-a)}^6 } |\te(0)|^2_{2,\Om} + c |\int_{\Om} \theta dx|^2 \\ \le c e^{c|d_1|_{3,6,S_2^t(-a)}^6 } |\te(0)|^2_{2,\Om} + c |d_1|^2_{1,S_2^t(-a)} {\theta^*}^2 + c |\te(0)|^2_{1,\Om} \eenn 

\noindent Therefore, this  implies (\ref{5.30}) and ends the proof.
\end{proof}

\section{Weak solutions}

\begin{theorem}\label{t4.1}
Let the assumptions of Theorem~\ref{l3.1} and Theorem~\ref{l5.8} hold. Then there exists a weak solution to problems (\ref{4.8})-(\ref{4.9}) and (\ref{1.1}).
\end{theorem}

\begin{proof}
We prove the existence through the Galerkin approximations. We set $\{\psi_i\}_{i=1}^\infty$ and $\{\phi_i\}_{i=1}^\infty$ be bases in $V_1$ and $V_2$, respectively. We choose the bases which are orthonormal in $L_2(\Omega)$ and orthogonal in $V_1$ and $V_2$, respectively. Next, for each $m\in\N$ we define an approximate solution $w_m(t)=\sum_{i=1}^mc_{im}(t)\psi_i$, $\vartheta_m(t)=\sum_{i=1}^md_{im}(t)\phi_i$, that satisfy the following system
\begin{equation}\eqal{
&\intop_\Omega w'_m(t)\cdot\psi_jdx+\intop_\Omega H(w_m(t))\cdot\psi_jdx+\nu \intop_\Omega\D(w_m(t))\cdot\D(\psi_j)dx\cr
&\quad+\gamma\sum_{\alpha=1}^2\intop_{S_1}w_m(t)\cdot\bar\tau_\alpha\psi_j\cdot \bar\tau_\alpha dS-\sum_{\alpha,\beta=1}^2\intop_{S_\beta}B_{\beta\alpha} (\delta)\psi_j\cdot\bar\tau_\alpha dS\cr
&=\intop_\Omega \omega(\vartheta_m(t))f(t)\cdot\psi_jdx+\intop_\Omega F(\delta,t)\cdot\psi_jdx\cr
&{\rm for}\ j=1,\cdots,m,\ {\rm a.a.}\ t\in[kT,(k+1)T],\ k\in\N_0,\cr}
\label{4-1}
\end{equation}
\begin{equation}\eqal{
&\intop_\Omega\vartheta'_m(t)\phi_jdx+\varkappa\intop_\Omega\nabla\vartheta_m(t) \cdot\nabla\phi_jdx\cr
&\quad+\intop_\Omega w_m(t)\cdot\nabla\vartheta_m(t)\phi_jdx=-\intop_\Omega \delta\cdot\nabla\vartheta_m(t)\phi_jdx\cr
&{\rm for}\ j=1,\cdots,m,\ {\rm a.a.}\ t\in[kT,(k+1)T],\ k\in\N_0,\cr}
\label{4-2}
\end{equation}
$$\eqal{
&w_m|_{t=kT}=w_{kTm},\cr
&\vartheta_m|_{t=kT}=\vartheta_{kTm},\cr}
$$
where $w_{kTm}\in{\rm span}\{\psi_1,\cdots,\psi_m\}$, $w_{kTm}\to w(kT)$ strongly in $L_2(\Omega)$ as $m\to\infty$, $\vartheta_{kTm}\in{\rm span}\{\phi_1,\cdots,\phi_m\}$, $\vartheta_{kTm}\to\vartheta(kT)$ strongly in $L_2(\Omega)$.

\noindent
We multiply (\ref{4-1}) by $c_{jm}$ and sum for $j=1,\cdots,m$. In analogous way, we multiply (\ref{4-2}) by $d_{jm}$ and sum for $j=1,\cdots,m$. Then we obtain relations (\ref{3.8}) and (\ref{3-5}) with $w$ and $\vartheta$ replaced by $w_m$ and $\vartheta_m$. For terms on the right-hand sides of these inequalities we can obtain the same estimates as in the proof of Lemma \ref{l3-1}. Finally, we conclude inequality (\ref{3.3}) for the Galerkin approximations $w_m$ and $\vartheta_m$ which yield the inequality
\begin{equation}\eqal{
&\|w_m\|_{L_\infty(kT,(k+1)T;L_2(\Omega))}+ \|w_m\|_{L_2(kT,(k+1)T;H^1(\Omega))}\cr
&\quad+\|\vartheta_m\|_{L_\infty(kT,(k+1)T;L_2(\Omega))}+ \|\vartheta_m\|_{L_2(kT,(k+1)T;H^1(\Omega))}\cr
&\le c(A_2(T))\quad {\rm for\ all}\ \ t\in(kT,(k+1)T),\ k\in\N_0,\cr}
\label{4-3}
\end{equation}
where $A_2=A_2(T)$ is described in (\ref{3.23}) which also holds for $w_m$ and $\vartheta_m$. 

\noindent
In order to pass to the limit in the second term of (\ref{4-1}) it is necessary to have estimates for  $\|w'_m\|_{L_{4/3}(kT,(k+1)T;V_1^*)}$ and $\|\vartheta'_m\|_{L_{8/7}(kT,(k+1)T;V_2^*)}$. Therefore, we apply Aubin-Lions Lemma (see \cite[Ch. 1, Sect. 5]{L}).

\noindent
Suppose that $u\in L_4(kT,(k+1)T;V_1)$ and  $\|u\|_{L_4(kT,(k+1)T;V_1)}\le 1$. We can express  $u(x,t)=\sum_{i=1}^\infty a_i(t)\psi_i(x)$, with $a_i(t)=\intop_\Omega u(x,t)\psi_ixdx$. Let
$$
u_m(x,t):=\sum_{k=1}^ma_k(t)\psi_k(x).
$$
Then
$$
\|u_m\|_{L_4(kT,(k+1)T;V_1)}\le\|u\|_{L_4(kT,(k+1)T;V_1)}.
$$
Thus, from (\ref{4-1})
$$\eqal{
&\bigg\|{\partial w_m\over\partial t}\bigg\|_{L_{4/3}(kT,(k+1)T;V_1^*)}= \sup_{u\in L_4(kT,(k+1)T;V_1)\atop\|u\|\le1}\left|\intop_{kT}^{(k+1)T}\intop_\Omega {\partial w_m\over\partial t}udxdt\right|\cr
&=\sup_{u\in L_4(kT,(k+1)T,V_1)\atop\|u\|\le1}\left|\intop_{kT}^{(k+1)T} \intop_\Omega{\partial w_m\over\partial t}u_mdxdt\right|\cr
&=\sup_{u\in L_4(kT,(k+1)T;V_1)\atop\|u\|\le1}\bigg|\intop_{kT}^{(k+1)T} \intop_\Omega \omega(\vartheta_m) u_m f dxdt\cr
&\quad+\intop_{kT}^{(k+1)T}\intop_\Omega F(\delta,t)u_mdxdt- \nu\intop_{kT}^{(k+1)T}\intop_\Omega\D(w_m)\cdot\D(u_m)dxdt\cr
&\quad-\gamma\sum_{\alpha=1}^2\intop_{kT}^{(k+1)T}\intop_{S_1}w_m\cdot \bar\tau_\alpha u_m\cdot\bar\tau_\alpha dS-\intop_{kT}^{(k+1)T}\intop_\Omega H(w_m)u_mdxdt\cr
&\quad+\sum_{\alpha,\beta=1}^2\intop_{kT}^{(k+1)T}\intop_{S_\beta} B_{\beta\alpha}(\delta)u_m\cdot\bar\tau_\alpha dS\bigg|.\cr}
$$
Applying definitions of $F(\delta,t)$ and $H(w)$ and the following relation 
$$
\intop_\Omega w_m\cdot\nabla w_m\cdot u_mdx=-\intop_\Omega w_m\cdot (w_m\cdot\nabla u_m)dx
$$
we obtain 

$$\eqal{
&\bigg\|{\partial w_m\over\partial t}\bigg\|_{L_{4/3}(kT,(k+1)T;V_1^*)}\le c\bigg(\sup_u\intop_{kT}^{(k+1)T}\|f\|_{L_{6/5}(\Omega)}\|u_m\|_{L_6(\Omega)}dt\cr
&\quad+\sup_u\intop_{kT}^{(k+1)T}\|\vartheta_m\|_{L_2(\Omega)}\|f\|_{L_3(\Omega)} \|u_m\|_{L_6(\Omega)}dt\cr
&\quad+\sup_u\intop_{kT}^{(k+1)T}\|\delta_t\|_{L_{6/5}(\Omega)} \|u_m\|_{L_6(\Omega)}dt\cr
&\quad+\sup_u\intop_{kT}^{(k+1)T}\|\nabla\delta\|_{L_2(\Omega)} \|\delta\|_{L_2(\Omega)}\|u_m\|_{L_6(\Omega)}dt\cr
&\quad+\sup_u\intop_{kT}^{(k+1)T}\|\D(\delta)\|_{L_2(\Omega)}\|\nabla u_m\|_{L_2(\Omega)}dt\cr
&\quad+\sup_u\intop_{kT}^{(k+1)T}\|\D(w_m)\|_{L_2(\Omega)} \|\D(u_m)\|_{L_2(\Omega)}dt\cr
}$$ $$\eqal{
&\quad+\sup_u\intop_{kT}^{(k+1)T}\|w_m\|_{L_2(S_1)}\|u_m\|_{L_2(S_1)}dt\cr
&\quad+\sup_u\intop_{kT}^{(k+1)T}\sum_{\alpha=1}^2 \|\delta\cdot\bar\tau_\alpha\|_{L_2(S_1)}\|u_m\|_{L_2(S_1)}dt\cr
&\quad+\sup_u\intop_{kT}^{(k+1)T}\|w_m\|_{L_4(\Omega)}^2\|\nabla u_m\|_{L_2(\Omega)}dt\cr
&\quad+\sup_u\intop_{kT}^{(k+1)T}\|w_m\|_{L_4(\Omega)}\|u_m\|_{L_4(\Omega)} \|\nabla\delta\|_{L_2(\Omega)}dt\cr
&\quad+\sup_u\intop_{kT}^{(k+1)T}\|\nabla w_m\|_{L_2(\Omega)}\|u_m\|_{L_4(\Omega)}\|\delta\|_{L_4(\Omega)}dt\bigg),\cr}
$$
where supremum is taken over all functions $u\in L_4(kT,(k+1)T;V_1)$ with $\|u\|_{L_4(kT,(k+1)T;V_1)}\le 1$.

\noindent
Next, we apply imbedding theorems and the interpolation inequality
\begin{equation}
\|w_m\|_{L_4(\Omega)}\le c\|w_m\|_{L_2(\Omega)}^{1/4}\|w_m\|_{V_1}^{3/4}
\label{4.4}
\end{equation}
to have
\begin{equation}\eqal{
&\bigg\|{\partial w_m\over\partial t}\bigg\|_{L_{4/3}(kT,(k+1)T;V_1^*)}\le c[\|f\|_{L_{4/3}(kT,(k+1)T;L_{6/5}(\Omega))}\cr
&\quad+\|f\|_{L_{4/3}(kT,(k+1)T;L_3(\Omega))} \|\vartheta_m\|_{L_\infty(kT,(k+1)T,L_2(\Omega))}\cr
&\quad+\|\delta_t\|_{L_{4/3}(kT,(k+1)T;L_{6/5}(\Omega))}+ \|\nabla\delta\|_{L_{4/3}(kT,(k+1)T;L_2(\Omega))}\cdot\cr
&\quad\cdot\|\delta\|_{L_\infty(kT,(k+1)T;L_3(\Omega))}+ \|\D(\delta)\|_{L_{4/3}(kT,(k+1)T;L_2(\Omega))}\cr
&\quad+\|\D(w_m)\|_{L_{4/3}(kT,(k+1)T;L_2(\Omega))}+ \|w_m\|_{L_{4/3}(kT,(k+1)T;L_2(S_1))}\cr
&\quad+\sum_{\alpha=1}^2 \|\delta\cdot\bar\tau_\alpha\|_{L_{4/3}(kT,(k+1)T;L_2(S_1))}+ \|w_m\|_{L_\infty(kT,(k+1)T;L_2(\Omega))}^{1/2}\cdot\cr
&\quad\cdot\|w_m\|_{L_2(kT,(k+1)T;V_1)}^{3/2}+ \|w_m\|_{L_2(kT,(k+1)T;V_1)}\cdot\cr
&\quad\cdot(\|\nabla\delta\|_{L_4(kT,(k+1)T;L_2(\Omega))}+ \|\delta\|_{L_4(kT,(k+1)T;L_4(\Omega))})].\cr}
\label{4-5}
\end{equation}

\noindent
To estimate the norms of functions $w_m$ and $\vartheta_m$ we apply (\ref{4.3}) and all norms of function $\delta$ are estimated by some norms of $d$.

\noindent
Then (\ref{4-5}) implies
\begin{equation}
\bigg\|{\partial w_m\over\partial t}\bigg\|_{L_{4/3}(kT,(k+1)T;V_1^*)}\le C,
\label{4-6}
\end{equation}
where $C=C(T,f,d)$ is a positive constant.

Now, assume that $\zeta\in L_8(kT,(k+1)T;V_2)$ and $\|\zeta\|_{L_8(kT,(k+1)T;V_2)}\le 1$. Then
$$
\zeta(x,t)=\sum_{i=1}^\infty b_i(t)\phi_i(x),\quad {\rm where}\quad b_i(t)=\intop_\Omega\zeta(x,t)\phi_i(x).
$$
Let
$$
\zeta_m(x,t)=\sum_{k=1}^mb_k(t)\phi_k(x).
$$
We have
\begin{equation}
\|\zeta_m\|_{L_8(kT,(k+1)T;V_2)}\le\|\zeta\|_{L_8(kT,(k+1)T;V_2)}
\label{4-7}
\end{equation}
Thus
$$\eqal{
&\bigg\|{\partial\vartheta_m\over\partial t}\bigg\|_{L_{8/7}(kT,(k+1)T;V_2^*)}= \sup_{\zeta\in L_8(kT,(k+1)T;V_2)\atop\|\zeta\|\le1}\bigg|\intop_{kT}^{(k+1)T} {\partial\vartheta_m\over\partial t}\zeta dxdt\bigg|\cr
&=\sup_{\zeta\in L_8(kT,(k+1)T;V_2)\atop\|\zeta\|\le1}\bigg|\intop_{kT}^{(k+1)T} \intop_\Omega{\partial\vartheta_m\over\partial t}\zeta_mdxdt\bigg|\cr
&=\sup_\zeta\bigg|-\varkappa\intop_{kT}^{(k+1)T}\intop_\Omega\nabla\vartheta_m \cdot\nabla\zeta_mdxdt\cr
&\quad-\intop_{kT}^{(k+1)T}\intop_\Omega w_m\cdot\nabla\vartheta_m\zeta dxdt- \intop_{kT}^{(k+1)T}\intop_\Omega\delta\cdot\nabla\vartheta_m\zeta_mdxdt\bigg| \cr
&\quad \le c\bigg(\sup_\zeta\intop_{kT}^{(k+1)T}\|\nabla\vartheta_m\|_{L_2(\Omega)} \|\nabla\zeta_m\|_{L_2(\Omega)}dt\cr
&\quad+\sup_\zeta\intop_{kT}^{(k+1)T}\|w_m\|_{L_4(\Omega)}\|\zeta_m\|_{L_4(\Omega)} \|\nabla\vartheta_m\|_{L_2(\Omega)}dt\cr
&\quad+\sup_\zeta\intop_{kT}^{(k+1)T}\|\nabla\vartheta_m\|_{L_2(\Omega)} \|\delta\|_{L_4(\Omega)}\|\zeta_m\|_{L_4(\Omega)}dt
\bigg).\cr}
$$
Therefore, with the imbedding theorem and inequality (\ref{4.4}) we get
\begin{equation}\eqal{
&\bigg\|{\partial\vartheta_m\over\partial t}\bigg\|_{L_{8/7}(kT,(k+1)T;V_2^*)}\le c[\sup_\zeta(\|\nabla\vartheta_m\|_{L_2(kT,(k+1)T;L_2(\Omega))}\cr
&\quad\cdot\|\nabla\zeta_m\|_{L_2(kT,(k+1)T;L_2(\Omega))})+\sup_\zeta (\|w_m\|_{L_\infty(kT,(k+1)T;L_2(\Omega))}^{1/4}\cr
&\quad\cdot\|\nabla\vartheta_m\|_{L_2(kT,(k+1)T;L_2(\Omega))}\|\nabla w_m\|_{L_2(kT,(k+1)T;L_2(\Omega))}^{3/8}\cr
&\quad\cdot\|\zeta_m\|_{L_8(kT,(k+1)T;L_2(\Omega))}^{1/8})+\sup_\zeta (\|\nabla\vartheta_m\|_{L_2(kT,(k+1)T;L_2(\Omega))}\cr
&\quad\cdot\|\delta\|_{L_4(kT,(k+1)T;L_4(\Omega))} \|\zeta_m\|_{L_4(kT,(k+1)T;L_4(\Omega))}) \cr}
\label{4-8}
\end{equation}
where supremum is taken over all functions $\zeta\in L_8(kT,(k+1)T;V_2)$ with $\|\zeta\|_{L_8(kT,(k+1)T;V_2)}\le 1$.

\noindent
Then, using (\ref{4-8}), (\ref{4-3}) and the definition of $\delta$ we obtain
\begin{equation}
\bigg\|{\partial\vartheta_m\over\partial t}\bigg\|_{L_{8/7}(kT,(k+1)T;V_2^*)}\le C
\label{4-9}
\end{equation}
with $C=C(T,f,d)$ - a positive constant.

\noindent
By (\ref{4-3}), (\ref{4-6}) and (\ref{4-9}) there exist subsequences of $(w_m)$ and $(\vartheta_m)$ still denoted by $(w_m)$ and $(\vartheta_m)$ such that
$$\eqal{
w_m\to w\ &{\rm weakly\ in}\ L_2(kT,(k+1)T;V_1),\cr
{\rm and}\ *-&{\rm weakly\ in}\ L_\infty(kT,(k+1)T;L_2(\Omega)),\cr
\noalign{\vskip 8pt}
\vartheta_m\to\vartheta\ &{\rm weakly\ in}\ L_2(kT,(k+1)T;V_2)\cr
{\rm and}\ *-&{\rm weakly\ in}\ L_\infty(kT,(k+1)T;L_2(\Omega)),\cr
\noalign{\vskip 8pt}
w_{mt}\to w_t\ &{\rm weakly\ in}\ L_{4/3}(kT,(k+1)T;V_1^*),\cr
\vartheta_{mt}\to\vartheta_t\ &{\rm weakly\ in}\ L_{8/7}(kT,(k+1)T;V_2^*).\cr}
$$
Moreover, the Aubin-Lions lemma yields that
$$\eqal{
w_m\to w\ &{\rm strongly\ in}\ L_2(kT,(k+1)T,L_2(\Omega))\cr
{\rm and}\ \vartheta_m\to\vartheta\ &{\rm strongly\ in}\ L_2(kT,(k+1)T;L_2(\Omega))\cr}
$$
for some subsequences $(w_m)$ and $(\vartheta_m)$ of sequences $(w_m)$, $(\vartheta_m)$.

\noindent
Consequently, by (\ref{4.3}) and above convergences in $L^2(kT,(k+1)T;L_2(\Omega))$, we obtain 
$$\eqal{
w_m\to w\ &{\rm strongly\ in}\ L_2(kT,(k+1)T;L_p(\Omega)),\cr
\vartheta_m\to\vartheta\ &{\rm strongly\ in}\ L_2(kT,(k+1)T;L_p(\Omega))\ {\rm for}\ p<6.\cr}
$$
With above convergences we are able to pass to the limit in (\ref{4-1})--(\ref{4-2}) in the standard way.  This yield that $(w,\vartheta)$ is the weak solution to (\ref{4.8})-(\ref{4.9}).

\noindent
\end{proof}

\end{document}